\pdfoutput=1
\documentclass[11pt,a4paper]{amsart}

\usepackage[utf8]{inputenc}

\usepackage[left=1in,right=1in,top=1in,bottom=1in]{geometry}

\usepackage{amssymb}
\usepackage{amsmath}
\usepackage{amsthm}
\usepackage{thmtools}

\usepackage{enumitem}

\usepackage[hypertexnames=false]{hyperref}

\usepackage[english,noabbrev,sort]{cleveref}

\usepackage{tikz}
\usetikzlibrary{decorations.pathreplacing}
\usetikzlibrary{backgrounds}
\usepackage{tikz-cd}

\usepackage{verbatim}

\let\C\undefined
\usepackage{my-math-definitions}

\numberwithin{equation}{section}

\newtheorem{theorem}{Theorem}[section]
\newtheorem{lemma}[theorem]{Lemma}
\newtheorem{corollary}[theorem]{Corollary}

\theoremstyle{remark}
\newtheorem*{remark}{Remark}
\newtheorem*{example}{Example}

\newtheorem{claim}{Claim}
\numberwithin{claim}{theorem}

\theoremstyle{definition}
\newtheorem{definition}[theorem]{Definition}

\author{Fabian Gundlach \and Jürgen Klüners}
\title{Symmetries of power-free integers in number fields and their shift spaces}

\renewcommand{\P}{\mathcal P}
\newcommand{\X}{\mathbb X}
\newcommand{\Y}{\mathbb Y}
\newcommand{\T}{\mathcal T}
\DeclareMathOperator{\ExSym}{ExSym}
\newcommand{\LinSym}{\mathcal H}
\DeclareMathOperator{\meas}{meas}

\begin{document}

\begin{abstract}
We describe the group of $\mathbb Z$-linear automorphisms of the ring of integers of a number field~$K$ that preserve the set $V_{K,k}$ of $k$th power-free integers: every such map is the composition of a field automorphism and the multiplication by a unit.

We show that those maps together with translations generate the extended symmetry group of the shift space $\mathbb D_{K,k}$ associated to $V_{K,k}$.
Moreover, we show that no two such dynamical systems $\mathbb D_{K,k}$ and $\mathbb D_{L,l}$ are topologically conjugate and no one is a factor system of another.

We generalize the concept of $k$th power-free integers to sieves and study the resulting admissible shift spaces.
\end{abstract}

\maketitle

\section{Introduction}
\label{introduction}

For any integer $k\geq2$, consider the set $V_{K,k}$ of \emph{$k$-free integers} (also called $k$th power-free integers) in a number field $K$, i.e.{} the set of integers $x\in\O_K$ that are not divisible by the $k$-th power of any prime ideal. For any prime number $p$, let $V_{K,k,p}$ be the set of residue classes modulo $p^k$ that are not divisible by the $k$-th power of any prime ideal dividing $p$.

\subsection{Symmetries}
In this paper, we compute the group of $\Z$-linear symmetries of the set $V_{K,k}$. More precisely, we show:

\begin{theorem}
\label{intro_symmetries}
Let $K$ be a number field and let $A:\O_K\ra\O_K$ be a $\Z$-linear bijection. Note that $A$ gives rise to a $\Z$-linear bijection $A:\O_K/p^k\O_K \ra \O_K/p^k\O_K$ for every prime number $p$. Then, the following are equivalent:
\begin{enumerate}[label=(\alph*)]
\item The map $A$ satisfies the global condition $A(V_{K,k})=V_{K,k}$.
\item The map $A$ satisfies the local condition $A(V_{K,k,p})=V_{K,k,p}$ for all prime numbers $p$.
\item The map $A$ is the composition $M_\varepsilon\circ\tau$ of a field automorphism $\tau$ of $K$ and the multiplication by $\varepsilon$ map $M_\varepsilon$ for some $\varepsilon\in\O_K^\times$ (given by $M_\varepsilon(x)=\varepsilon x$).
\end{enumerate}
\end{theorem}
Note that the group of bijections satisfying (c) is isomorphic to $\O_K^\times\rtimes\Aut(K)$, where $\Aut(K)$ is the automorphism group of the number field $K$. (We do not assume $K$ to be a normal extension of $\Q$.)

The implication (c) $\Rightarrow$ (a) is clear. The implication (a) $\Rightarrow$ (b) follows immediately from a \emph{local-global principle} for $k$-free integers, namely the surjectivity of the reduction modulo $p^k$ map $V_{K,k}\ra V_{K,k,p}$ for $k\geq2$. (See \Cref{cor_local_global_kfree}.) The implication (b) $\Rightarrow$ (c) (see \Cref{cor_linear_maps}), which we show in \Cref{section_thm_main}, is the heart of this paper. (The implications (b) $\Rightarrow$ (a) and (c) $\Rightarrow$ (b) are also clear. The implication (a) $\Rightarrow$ (c) is harder.)

With the same method as before, we prove:

\begin{theorem}[see \Cref{cor_local_global_kfree,cor_linear_maps}]
\label{intro_linear_maps}
Let $K,L$ be number fields and let $A:\O_K\ra\O_L$ be a $\Z$-linear map. Then, the following statements satisfy the implications (a) $\Leftrightarrow$ (b) $\Rightarrow$ (c):
\begin{enumerate}[label=(\alph*)]
\item The map $A$ satisfies the global condition $A(V_{K,k})\subseteq V_{L,l}$.
\item The map $A$ satisfies the local condition $A(V_{K,k,p})\subseteq V_{L,l,p}$ for all prime numbers $p$.
\item The map $A$ is the composition $M_\varepsilon\circ\tau$ of a field homomorphism $\tau:K\ra L$ and the multiplication by $\varepsilon$ map $M_\varepsilon$ for some $\varepsilon\in\O_L$.
\end{enumerate}
\end{theorem}

Note that in (a) and (b), we only require inclusion, not equality.

\begin{remark}
The implication (c) $\Rightarrow$ (a) can fail. For example, take the inclusion $A=M_1\circ\id:\Z\hookrightarrow\Z[\sqrt3]$. We have $3\in V_{\Q,2}$ but $3=\sqrt{3}^2\notin V_{\Q(\sqrt3),2}$.
\end{remark}

\subsection{Shift spaces}

Following \cite{positive-entropy-shifts}, we can associate an $n$-dimensional shift space to the set of $k$-free integers in a number field $K$ of degree $n$: Let the group $\O_K\cong\Z^n$ act by translation on the set $\P(\O_K)$ of subsets of $\O_K$. This action is continuous with respect to the product topology on $\P(\O_K)\cong\{0,1\}^{\O_K}$. The \emph{shift space $\D_{K,k}$ associated to the $k$-free integers} is the topological dynamical system given by the (continuous) action of $\O_K$ on a closed subset $\X_{K,k}$ of $\P(\O_K)$. The set $\X_{K,k}$ can be defined in two ways:
\begin{itemize}
\item \emph{Globally}, as $\X_{K,k}:=\overline{\O_K+V_{K,k}}$, the closure of the orbit of $V_{K,k}\in\P(\O_K)$ under the action of $\O_K$.
\item \emph{Locally}, as the set of \emph{admissible} subsets $S$ of $\O_K$: subsets that miss at least one residue class modulo the $k$-th power of every prime ideal.
\end{itemize}

The equivalence of these two definitions (another local-global principle) was observed in \cite{positive-entropy-shifts}. We give a detailed proof of the equivalence in \Cref{thm_orbit_closure}.

Using a variant of \Cref{intro_symmetries}, we compute the extended symmetry group of the dynamical systems defined above (the normalizer of $\O_K$ in the homeomorphism group of $\X_{K,k}$):

\begin{theorem}[see \Cref{thm_automorphisms}]
\label{intro_automorphisms}
The extended symmetry group of the topological dynamical system $\D_{K,k}$ consists of the pairs $(f,A)$, where $A:\O_K\ra\O_K$ is of the form $A=M_\varepsilon\circ\tau$ with $\tau\in\Aut(K)$ and $\varepsilon\in\O_K^\times$, and where $f:\X_{K,k}\ra\X_{K,k}$ is defined by $f(S)=t+\varepsilon\cdot\tau(S)$ with $t\in\O_K$. In particular, the extended symmetry group is isomorphic to $\O_K\rtimes(\O_K^\times\rtimes\Aut(K))$.
\end{theorem}

We show that no two shift spaces of the form $\D_{K,k}$ are topologically conjugate, and more generally that no one is a factor system of another:

\begin{theorem}\label{thm_kfree_shift_spaces_cong}
\label{intro_conjugate}
Let $K,L$ be number fields and let $k,l\geq2$. Then, the following are equivalent:
\begin{enumerate}[label=(\alph*)]
\item The topological dynamical systems $\D_{K,k}$ and $\D_{L,l}$ are topologically conjugate.
\item The topological dynamical system $\D_{L,l}$ is a factor system of $\D_{K,k}$.
\item We have $K\cong L$ and $k=l$.
\end{enumerate}
\end{theorem}
We show \Cref{thm_kfree_shift_spaces_cong} directly without explicitly making use of any invariants of the dynamical systems.
The implications (c) $\Rightarrow$ (a) $\Rightarrow$ (b) are clear. The proof of (b) $\Rightarrow$ (c) (see \Cref{thm_factor_kfree}) relies on a generalization of \Cref{intro_linear_maps}. 

\subsection{Known results}

\Cref{intro_symmetries,intro_automorphisms} were previously known for quadratic number fields (see \cite[Theorems 3.11, 3.15, and 4.4]{baake-bustos-nickel-power-free-points}) and cyclotomic fields (see \cite[Theorems 5.9 and 6.2]{k-free-cyclotomic}). These results are not sufficient to distinguish in all situations different number fields or distinct exponents like in \Cref{thm_kfree_shift_spaces_cong}.

In the quadratic case some of the shift spaces $\D_{K,k}$ had previously been distinguished by their extended symmetry group.
The quotient of the extended symmetry group modulo translations by elements of $\O_K$ is $\O_K^\times\rtimes\Aut(K)$. This quotient group is finite for imaginary quadratic number fields and infinite for real quadratic number fields. Hence, $\D_{K,k}$ and $\D_{L,l}$ cannot be topologically conjugate when $K$ and $L$ are quadratic number fields of different signature.

In case both unit groups are ~$\{\pm1\}$ this approach does not help.
Another invariant that had been used to distinguish dynamical systems is their topological entropy. The topological entropy of $\D_{K,k}$ equals its patch-counting entropy, which is $\log(2)/\zeta_K(k)$. (See \cite[Theorem~5.1]{baake-bustos-nickel-power-free-points} for the case of quadratic number fields.)

Unfortunately, the patch-counting entropy is not a complete invariant of the topological dynamical systems $\D_{K,k}$ up to topological conjugacy:
$\zeta_K(k)=\zeta_L(l)$ does not necessarily imply that $K\cong L$: there are \emph{arithmetically equivalent} number fields $K\ncong L$, which have exactly the same zeta function $\zeta_K(s)=\zeta_L(s)$ for all $s$. (See \cite{perlis-arithmetic-equivalence}.)

Nevertheless, there are many cases in which the patch-counting entropy suffices to prove that two dynamical systems are not topologically conjugate. Known results are:

\begin{itemize}
\item If $K=L$ and $1\leq k<l$, then clearly $1/\zeta_K(k)<1/\zeta_L(l)$.
\item If $K$ and $L$ are real quadratic number fields and $k,l\geq2$ are even, then $\zeta_K(k)=\zeta_L(l)$ implies $K\cong L$ and $k=l$. (See \cite[Lemma 5.2]{baake-bustos-nickel-power-free-points}.)
\item If $K$ and $L$ are imaginary quadratic number fields and $k,l\geq3$ are odd, then $\zeta_K(k)=\zeta_L(l)$ implies $K\cong L$ and $k=l$, assuming that the numbers $\pi,\zeta(3),\zeta(5),\dots$ are algebraically independent. (See \cite[Corollary 5.4 and the discussion preceding it]{baake-bustos-nickel-power-free-points}.)
\end{itemize}

For other contexts in which the groups of linear maps preserving certain sets have been determined, see for example \cite{li-tsing-linear-preserver-problems,li-pierce-linear-preserver-problems}. Inspired by our work, Seguin \cite{seguin-sympol} recently treated questions analogous to \Cref{intro_symmetries}, replacing the ring of integers $\O_K$ by the polynomial ring over a field.

\subsection{Structure of the paper}

In \Cref{section_sieves}, we define sieves, which generalize the notion of $k$-freeness. This generalization is not only natural from a number-theoretic perspective, but also helpful in the proof of the aforementioned results about shift spaces, especially of (b) $\Rightarrow$ (c) in \Cref{intro_conjugate}.
In \Cref{section_local_global}, we show a local-global principle for $k$-free integers, yielding the implications (a) $\Rightarrow$ (b) in \Cref{intro_symmetries,intro_linear_maps}.
In \Cref{section_thm_main}, we state and prove (a generalization to sieves of) the implications (b) $\Rightarrow$ (c) in \Cref{intro_symmetries,intro_linear_maps}.
In \Cref{section_dynamical_systems}, we define and study shift spaces associated to sieves and prove \Cref{intro_automorphisms,intro_conjugate}. While our results on morphisms between shift spaces associated to $k$-free numbers can be considered satisfactory, some natural questions for other sieves remain open. (See \Cref{section_morphism_open_questions}.)

We have excluded the case $k=1$ above, which would give rise to the set of units, $V_{K,1}=\O_K^\times$. For $k=1$, the local-global principles mentioned above already break down for $K=\Q$. (Not every non-zero residue class modulo a prime is the remainder of a unit.) In \Cref{section_units}, we study this case $k=1$ and show a statement similar to \Cref{intro_symmetries,intro_linear_maps} when $K$ is a totally real number field.

\subsection{Acknowledgements} This work was supported by the Deutsche Forschungsgemeinschaft (DFG, German Research Foundation) – Project-ID 491392403 – TRR 358. The authors are grateful to Francisco Araújo and Michael Baake for helpful discussions and comments on an earlier draft, and to the referee for the careful reading and helpful remarks.

\section{Sieves}\label{section_sieves}

The statements from the introduction can be generalized in various directions. Firstly, instead of a number field $K$, we can allow an arbitrary étale $\Q$-algebra $K$ (i.e., a cartesian product $K_1\times\cdots\times K_r$ of finitely many number fields) without complication. We will use the following notation:

\begin{definition}
Let $K$ be an étale $\Q$-algebra.
We denote its ring of integers by $\O_K$ and the completions of $K$ and $\O_K$ at a prime $\p$ by $K_\p$ and $\O_{K,\p}$, respectively. For any prime number $p$, we let $\O_{K,p}=\O_K\otimes_\Z\Z_p$. By the Chinese remainder theorem, we can identify $\O_{K,p}=\varprojlim_{k\geq0}\O_K/p^{k}\O_K$ with the product $\prod_{\p\mid p\textnormal{ prime of }K}\O_{K,\p}$. We naturally obtain a diagonal embedding $\O_K \hookrightarrow \prod_p \O_{K,p} = \prod_\p \O_{K,\p}$.
\end{definition}

To show results about the dynamical systems $\D_{K,k}$, a further generalization will be helpful.
The set $V_{K,k}$ of $k$-free integers is defined by the congruence conditions $x\nequiv0\mod\p^k$ for the primes $\p$ of $K$. In other words, we exclude the residue class $0\bmod\p^k$. More generally, we can for each prime $\p$ exclude another residue class, vary the exponent in the modulus $\p^{k(\p)}$, or even exclude finitely many arbitrary residue classes modulo some power $\p^{k(\p)}$. The compact open subsets of $\O_{K,\p}$ are exactly the unions of finitely many residue classes modulo a power of $\p$. These are also exactly the closed open subsets of $\O_{K,\p}$. This motivates the following definition:

\begin{definition}
Let $K$ be an étale $\Q$-algebra. We say that a compact open subset of $\O_{K,\p}$ is \emph{defined modulo} an invertible ideal $\a$ of $\O_K$ if it is a union of residue classes modulo $\a\O_{K,\p}$.

A \emph{sieve for $K$} is a collection $R=(R_\p)_\p$ of compact open subsets $R_\p\subseteq\O_{K,\p}$ for all primes $\p$ of $K$.
We let $V(K,R)$ be the set of $x\in\O_K$ with $x\notin R_\p$ for all primes $\p$. For any prime number~$p$, we let $V_p(K,R)$ be the subset $\prod_{\p\mid p}\O_{K,\p}\setminus R_\p$ of $\O_{K,p}=\prod_{\p\mid p}\O_{K,\p}$.
\end{definition}

\begin{remark}
By definition, an element $x$ of $\O_K$ lies in $V(K,R)$ if and only if it lies in $V_p(K,R)$ for all prime numbers $p$.
\end{remark}

\begin{example}
For any integer $k\geq1$, the \emph{$k$-free sieve} given by $R_\p=\p^k$ gives rise to the set $V_{K,k}=V(K,R)$ from the introduction and $V_p(K,R)$ is the union of the residue classes modulo $p^k$ contained in $V_{K,k,p}$.
\end{example}

\begin{example}
More generally, motivated by the notion of $\mathcal B$-free numbers (see for example \cite[paragraph before Definition~5.1]{positive-entropy-shifts} or \cite[Section 1.2]{dynamical-b-free}), we can fix an integer $k(\p)\geq1$ for each prime $\p$ and define $R_\p=\p^{k(\p)}$.
\end{example}

\begin{remark}
Like $V_{K,k}$, the set $V(K,R)$ can also be defined as a cut-and-project set (see \cite{jasjan-keller-lemanczyk-view-through-window} or \cite[Section 7.2]{baake-grimm}) as follows: Consider the locally compact abelian groups $G=K\otimes\R$ and $H = \prod_\p\O_{K,\p}$, the discrete co-compact subgroup $\mathcal L = \{(x,x,\dots) \mid x\in \O_K\}$ of $G\times H$, and the \emph{window} $W = \prod_\p (\O_{K,\p}\setminus R_\p)$. Then, $V(K,R)$ is the image of $(G\times W)\cap\mathcal L$ under the projection to $G$.
\end{remark}

It will often be important that the sieve does not exclude too many residue classes. To this end, we introduce the Haar measure on $\O_{K,\p}$, normalized so that $\meas(\O_{K,\p})=1$. The measure of any residue class modulo $\p^k$ is then $\meas(\p^k) = 1/\Nm(\p)^k$. The measure of a compact open subset $A$ of $\O_{K,\p}$ defined modulo $\p^k$ is the number of residue classes modulo $\p^k$ represented in~$A$ divided by the total number of residue classes modulo $\p^k$.

\begin{figure}[h]
\begin{tikzpicture}[decoration={brace,amplitude=4pt}]
\def\width{2.25}
\def\sep{2}
\def\height{0.35}
\newcommand{\drawbox}[2]{
	\begin{scope}[shift={({(\width+\sep)*#2},0)}]
		\def\prime{#1}
		\def\fr{(\prime*\prime)}
		\def\split{\width/\fr}
		\node[] at ({\width/2},{\height+0.3}) {$\O_{\Q,\prime} = \Z_\prime$};
		\draw[gray,fill=lightgray] (0,0) rectangle ({\split},\height);
		\foreach \j in {1,...,\prime} {
			\foreach \i in {1,...,\prime} {
				\draw ({\width*((\i-1)+(\j-1)*\prime)/\fr}, 0) -- +(0,0.06);
			}
			\draw ({\width*((\j-1)*\prime)/\fr}, 0) -- +(0,0.15);
		}
		\draw[|-|] ({\split},-0.15)
			-- node[below] {\clap{$R_\prime$}}
			(0,-0.15);
		\draw[|-|] (\width,-0.15)
			-- node[below] {$V_{\prime}(\Q,R)$}
			({\split},-0.15);
		\draw (0,0) rectangle (\width,\height);
	\end{scope}
}
\drawbox{2}{0}
\drawbox{3}{1}
\drawbox{5}{2}
\node at ({(\width+\sep)*3},{\height/2}) {$\cdots$};
\end{tikzpicture}
\caption{
The squarefree sieve $R_\p=\p^2$ for $K=\Q$. Each of the boxes represents one of the rings $\Z_p$. The ring $\Z_p$ is subdivided into residue classes modulo $p$, which are each subdivided into residue classes modulo $p^2$. The shaded area represents the excluded residue class $0\bmod p^2$. Each of the Haar measures is proportional to the area.
}
\end{figure}

\begin{figure}[h]
\begin{tikzpicture}[decoration={brace,amplitude=4pt}]
\def\width{2.25}
\def\sep{2}
\def\height{0.35}
\newcommand{\drawbox}[3]{
	\begin{scope}[shift={({(\width+\sep)*#2},0)}]
		\def\prime{#1}
		\def\grays{#3}
		\def\fr{(\prime*\prime)}
		\def\split{\width/\fr}
		\node[] at ({\width/2},{\height+0.3}) {$\O_{\Q,\prime} = \Z_\prime$};
		\foreach \beg / \en in \grays {
			\draw[gray,fill=lightgray] ({\beg*\width/\fr},0) rectangle ({\en*\width/\fr},\height);
		}
		\foreach \j in {1,...,\prime} {
			\foreach \i in {1,...,\prime} {
				\draw ({\width*((\i-1)+(\j-1)*\prime)/\fr}, 0) -- +(0,0.06);
			}
			\draw ({\width*((\j-1)*\prime)/\fr}, 0) -- +(0,0.15);
		}
		\draw (0,0) rectangle (\width,\height);
	\end{scope}
}
\drawbox{2}{0}{1/2}
\drawbox{3}{1}{0/2, 4/5, 6/7}
\drawbox{5}{2}{3/5, 6/8, 15/16, 20/21}
\node at ({(\width+\sep)*3},{\height/2}) {$\cdots$};
\end{tikzpicture}
\caption{
Another sieve for $K=\Q$.
}
\end{figure}

\begin{figure}[h]
\begin{tikzpicture}
\def\width{2.25}
\def\sep{2}
\newcommand{\drawticks}[2]{
	\def\nrticks{#1}
	\def\len{#2}
	\foreach \i in {1,...,\nrticks}{
		\draw ({\i*\width/(\nrticks+1)},0) -- +(0,\len);
		\draw (0,{\i*\width/(\nrticks+1)}) -- +(\len,0);
	};
}
\newcommand{\drawsquare}[3]{
	\def\labe{#1}
	\def\largeticks{#2}
	\def\smallticks{#3}
	\node[] at ({\width/2},{\width+0.3}) {\labe};
	\drawticks{\smallticks}{0.06}
	\drawticks{\largeticks}{0.15}
	\draw (0,0) rectangle (\width,\width);
}
\begin{scope}
	\draw[gray,fill=lightgray] (0,0) rectangle ({\width/4},{\width/4});
	\drawsquare{$\O_{K,2}$ (inert)}{1}{3}
	\node (R2) at ({\width/8},-0.75) {\clap{$R_{(2)}$}};
	\draw[->] (R2) -- ({\width/8},{\width/8});
	\node at ({\width*2/4+0.1},{\width*2/4+0.2}) {$V_2(K,R)$};
\end{scope}
\begin{scope}[shift={({\width+\sep},0)}]
	\fill[lightgray] (0,0) rectangle ({\width/9},{\width});
	\fill[lightgray] (0,0) rectangle ({\width},{\width/9});
	\draw[gray] (0,0) rectangle ({\width/9},{\width});
	\draw[gray] (0,0) rectangle ({\width},{\width/9});
	\drawsquare{$\O_{K,3}$ (split)}{2}{8}
	\draw[|->] (0,-0.25) -- node[below] {$\O_{K,\p_1}$} (\width,-0.25);
	\draw[|->] (-0.25,0) -- node[left] {$\O_{K,\p_2}$} (-0.25,\width);
	\draw[|-|] (0,-0.5) -- node[below] {$R_{\p_1}$} ({\width/9},-0.5);
	\draw[|-|] (-0.5,0) -- node[left] {$R_{\p_2}$} (-0.5,{\width/9});
	\node at ({\width*2/4+0.15},{\width*2/4+0.1}) {$V_3(K,R)$};
\end{scope}
\begin{scope}[shift={({(\width+\sep)*2},0)}]
	\draw[gray,fill=lightgray] (0,0) rectangle ({\width/25},{\width/25});
	\drawsquare{$\O_{K,5}$ (inert)}{4}{24}
	\node (R5) at ({\width/50},-0.75) {$R_{(5)}$};
	\draw[->] (R5) -- ({\width/50},{\width/50});
	\node at ({\width*2/4},{\width*2/4}) {$V_5(K,R)$};
\end{scope}
\node at ({(\width+\sep)*3},{\width/2}) {$\cdots$};
\end{tikzpicture}
\caption{
The squarefree sieve $R_\p=\p^2$ for
$K=\Q(\sqrt{13})$%
. Each of the boxes represents one of the rings $\O_{K,p}$. The prime numbers~$2$ and~$5$ are inert in $K$. The prime number $3$ splits into two primes $\p_1,\p_2$, so that $\O_{K,3}=\O_{K,\p_1}\times\O_{K,\p_2}$.
The shaded area represents the excluded residue classes. Each of the Haar measures is proportional to the area.
Note that the set $V_p(K,R)$ ``has more symmetries'' when $p$ is inert than when $p$ splits.
}
\label{fig_sqfree_sieve_nf}
\end{figure}

\section{Local-global principle}\label{section_local_global}

In this section, we state and prove a local-global principle for $k$-free integers and use it to prove the implications (a) $\Rightarrow$ (b) in \Cref{intro_symmetries,intro_linear_maps}. In \Cref{section_dynamical_systems}, the same local-global principle will be used to show the equivalence of the local and global definitions of $\X_{K,k}$.

We have the following commutative diagram, where the horizontal maps are the diagonal embeddings.
\[
\begin{tikzcd}
\O_K \rar[hook] & \prod_p \O_{K,p} \rar[equal] & \prod_\p\O_{K,\p} \\
V(K,R) \rar[hook] \uar[hook] & \prod_p V_p(K,R) \uar[hook] \rar[equal] & \prod_\p\O_{K,\p}\setminus R_\p \uar[hook]
\end{tikzcd}
\]

\subsection{Statement}

We first recall the following case of the classical strong approximation theorem:

\begin{theorem}
Let $K$ be an étale $\Q$-algebra.
For all finite sets $\Omega$ of primes of $K$, all integers $k\geq0$, and all tuples $x=(x_\p)_\p\in\prod_{\p\in\Omega}\O_{K,\p}$, there is an element $y$ of $\O_K$ satisfying $y \equiv x_\p \mod \p^k$ for all $\p\in\Omega$.
\end{theorem}
\begin{proof}
This follows immediately from the Chinese remainder theorem.
\end{proof}

\begin{remark}
The theorem can also be formulated topologically: Every non-empty open subset $U$ of $\prod_\p\O_{K,\p}$ contains an element of $\O_K$.
\end{remark}

Now, given a sieve $R$ for $K$, replace the condition $y\in\O_K$ in the strong approximation theorem by the stronger condition $y\in V(K,R)$. Note that $V(K,R)$ is defined by infinitely many congruence conditions, so the Chinese remainder theorem by itself is no longer sufficient. We show the following variant of the strong approximation theorem:

\begin{theorem}\label{thm_sqfree_local_global}
Let $R$ be a sieve for $K$. Assume that there is a finite subset $T$ of $\O_K$ such that for all primes $\p$, we have $R_\p\subseteq T+\p^2$ and $R_\p\neq\O_{K,\p}$.
For all finite sets $\Omega$ of primes of $K$, all integers $k\geq0$, and all tuples $x=(x_\p)_\p\in\prod_{\p\in\Omega}\O_{K,\p}\setminus R_\p$, there is an element $y$ of $V(K,R)$ satisfying $y \equiv x_\p \mod \p^k$ for all $\p\in\Omega$.
\end{theorem}

This theorem was at least implicitly stated in \cite[Proposition~5.2]{positive-entropy-shifts} (in the case of number fields). In our proof in \Cref{section_local_global_proof}, we pay attention to carefully estimate the tail term. A useful application of \Cref{thm_sqfree_local_global} is the following example:

\begin{example}
For $k\geq2$, the $k$-free sieve satisfies the assumptions of \Cref{thm_sqfree_local_global} with $T=\{0\}$. Taking $\Omega$ to be the set of primes~$\p$ of~$K$ dividing~$p$, we conclude that every $k$-free residue class modulo~$p^k$ contains a $k$-free number. Hence, the reduction modulo~$p^k$ map $V_{K,k}\ra V_{K,k,p}$ is surjective.
\end{example}
\begin{example}
On the other hand, consider the ``$1$-free sieve'' given by $R_\p=\p$. For simplicity, let $K=\Q$. The assumption of \Cref{thm_sqfree_local_global} fails as there is no finite subset $T$ of $\Z$ with $p\Z\subseteq T+p^2\Z$ for all prime numbers $p$. The conclusion fails as there is no element $y$ of $V(K,R)=\Z^\times=\{\pm1\}$ with $y \equiv 2 \mod 5$.
\end{example}

\begin{remark}
For $\mathcal B$-free systems, the local-global principle follows from the \emph{Erdős condition}: the convergence of $\sum_{\mathfrak b\in\mathcal B}\frac1{\Nm(\b)}<\infty$. It would be natural to say that a sieve satisfies the Erdős condition if $\sum_\p\meas(R_\p)<\infty$ and $R_\p\neq\O_{K,\p}$ for all primes $\p$. However, for general sieves, this condition does not imply the above local-global principle.

In fact, the following classical example shows that it is crucial that the set $T$ in \Cref{thm_sqfree_local_global} does not depend on the prime $\p$: Consider any enumeration $\p_1,\p_2,\dots$ of the primes of $K$ and any enumeration $a_1,a_2,\dots$ of the integers in $\O_K$. Let $R_{\p_i}=a_i+\p_i^2$. We have $R_\p\neq\O_{K,\p}$ for all primes $\p$. On the other hand, the set $V(K,R)$ is empty: for any $x\in\O_K$, there is some index $i$ with $a_i=x$ and we then have $x\in R_{\p_i}$ and therefore $x\notin V(K,R)$. Note that $\sum_\p\meas(R_\p)=\sum_\p\Nm(\p)^{-2}<\infty$. From the Chinese remainder theorem, one could naively expect the set $V(K,R)$ to have density $\prod_\p(1-\meas(R_\p))=\prod_\p(1-\Nm(\p)^{-2})=\zeta_K(2)^{-1}>0$ in $\O_K$.
\end{remark}

\subsection{Application to \texorpdfstring{$k$}{k}-free integers}

Before showing \Cref{thm_sqfree_local_global}, we show that it proves the implication (a) $\Rightarrow$ (b) in \Cref{intro_symmetries,intro_linear_maps}:

\begin{corollary}\label{cor_local_global_kfree}
Let $K,L$ be étale $\Q$-algebras and let $k\geq2$ and $l\geq1$. Let $A:\O_K\ra\O_L$ be a $\Z$-linear map. Then, the following are equivalent:
\begin{enumerate}[label=(\alph*)]
\item The map $A$ satisfies the global condition $A(V_{K,k})\subseteq V_{L,l}$.
\item The map $A$ satisfies the local condition $A(V_{K,k,p})\subseteq V_{L,l,p}$ for all prime numbers $p$.
\end{enumerate}
\end{corollary}

\begin{proof}
\begin{description}
\item[(b) $\Rightarrow$ (a)] Let $x \in V_{K,k}$. For any prime number $p$, we have $x \in V_{K,k,p}$ and therefore by assumption $A(x)\in V_{L,l,p}$. This implies $x\in V_{L,l}$.
\item[(a) $\Rightarrow$ (b)] We have seen that \Cref{thm_sqfree_local_global} implies the surjectivity of the map $V_{K,k}\ra V_{K,k,p}$. Let $\overline x\in V_{K,k,p}$ and let $x\in V_{K,k}$ be any preimage. By assumption, $A(x)\in V_{L,l}$, so in particular $A(\overline x)\in V_{L,l,p}$.
\qedhere
\end{description}
\end{proof}

\subsection{Proof}\label{section_local_global_proof}

We now prove the variant of strong approximation claimed above.

\begin{proof}[Proof of \Cref{thm_sqfree_local_global}]
We reduce to the case $\Omega=\emptyset$ as follows:

We need to prove that there is an element $y\in\O_K$ such that $y\notin R_\p$ for all primes~$\p$ and $y\equiv x_\p\mod\p^k$ for the finitely many primes $\p\in\Omega$. This is equivalent to $y\in V(K,R')$ for the sieve $R'=(R'_\p)_\p$ given by
\[
R'_\p =
\begin{cases}
R_\p,&\p\notin\Omega,\\
R_\p\cup(\O_{K,\p}\setminus(x_\p+\p^k)),&\p\in\Omega.
\end{cases}
\]
There is a finite subset $T'$ of $\O_K$ such that for all $\p$, we have $R'_\p\subseteq T'+\p^2$. (For each of the finitely many primes $\p\in\Omega$, add to the set $T$ representatives of all residue classes modulo $\p^2$, so that $T+\p^2=\O_{K,\p}$.)

Moreover, $x_\p\notin R'_\p$ for all primes $\p\notin\Omega$, so in particular $R'_\p\neq\O_{K,\p}$. Applying part \ref{lemma_sqfree_density_infinite} of the following lemma to the sieve $R'$, we obtain an element $y\in V(K,R')$ as claimed.
\end{proof}

\begin{lemma}\label{lemma_sqfree_density}
Let $R$ and $T$ as in \Cref{thm_sqfree_local_global}.
\begin{enumerate}[label=(\alph*)]
\item\label{lemma_sqfree_density_density}
Consider any $\Q$-vector space norm $\|\cdot\|$ on $K$. The \emph{density} of $V(K,R)$ in $\O_K$ is
\[
\lim_{X\ra\infty} \frac{\#\{x\in V(K,R):\|x\|\leq X\}}{\#\{x\in\O_K:\|x\|\leq X\}}
= \prod_{\p\textnormal{ prime of }K} (1-\meas(R_\p)) > 0.
\]
\item\label{lemma_sqfree_density_infinite}
In particular, the set $V(K,R)$ is infinite.
\end{enumerate}
\end{lemma}

To prove this lemma, we use a standard sieve theory argument. We will need the following tail estimate:

\begin{lemma}\label{tail_estimate}
Let $K$ be an étale $\Q$-algebra of degree $n$ and let $k\geq2$. For all $X,M\geq1$, we have the following uniform estimate:
\begin{align*}
N'(X,M) &:= \#\{x\in\O_K:\|x\|\leq X\textnormal{ and }\a^k\mid x\textnormal{ for some ideal $\a$ with }\Nm(\a)>M\}\\
&\ll 1 + \frac{X^n}{M^{k-1}}.
\end{align*}
(The error term can depend on $K$, $k$, and the norm $\|\cdot\|$.)
\end{lemma}

Since $\lim_{M\to\infty}\limsup_{X\to\infty}(1+X^n/M^{k-1})/X^n = \lim_{M\to\infty}1/M^{k-1} = 0$, this lemma shows that the $\mathcal B$-free system of $k$-free numbers has \emph{light tails}.

We first explain how to deduce \Cref{lemma_sqfree_density} from \Cref{tail_estimate}:

\begin{proof}[Proof of \Cref{lemma_sqfree_density}]
\begin{enumerate}[label=(\alph*)]
\item Let
\[
N(X) := \#\{x\in\O_K:\|x\|\leq X\textnormal{ and }x\notin R_\p\textnormal{ for all primes }\p\}.
\]
For any $M>1$, let
\[
N(X,M) := \#\{x\in\O_K:\|x\|\leq X\textnormal{ and }x\notin R_\p\textnormal{ for all primes $\p$ with $\Nm(\p)\leq M$}\}.
\]
Since there are only finitely many primes $\p$ with $\Nm(\p)\leq M$, we can apply the Chinese remainder theorem to estimate the size of $N(X,M)$. We obtain
\begin{align}\label{mainterm}
\lim_{X\ra\infty}\frac{N(X,M)}{\#\{x\in\O_K:\|x\|\leq X\}} = \prod_{\p:\Nm(\p)\leq M} (1-\meas(R_\p)).
\end{align}
Using the assumption that $R_\p\subseteq T+\p^2$, we get
\[
\arraycolsep=0pt
\renewcommand{\arraystretch}{1.3}
\begin{array}{rcl}
0
&\leq& N(X,M) - N(X) \\
&\leq& \#\{x\in\O_K:\|x\|\leq X\textnormal{ and }x\in R_\p\textnormal{ for some $\p$ with $\Nm(\p)> M$}\} \\
&\leq& \#\{x\in\O_K:\|x\|\leq X\textnormal{ and }\p^2\mid (x-t)\textnormal{ for some }t\in T\textnormal{ and $\p$ with $\Nm(\p)> M$}\} \\
&\underset{y=x-t}\leq&\displaystyle \sum_{t\in T} \#\{y\in\O_K:\|y\|\leq X+\|t\|\textnormal{ and }\p^2\mid y\textnormal{ for some $\p$ with $\Nm(\p)> M$}\} \\
&\underset{\textnormal{(Lem.\ \ref{tail_estimate})}}\ll&\displaystyle 1 + \frac{X^n}{M}.
\end{array}
\]
As $\#\{x\in\O_K:\|x\|\leq X\}\asymp X^n$, we conclude that
\[
\limsup_{X\ra\infty} \frac{|N(X,M)-N(X)|}{\#\{x\in\O_K:\|x\|\leq X\}} \ll \frac{1}{M}.
\]
Combining this with (\ref{mainterm}) and letting $M$ go to infinity (after $X$), we get
\[
\lim_{X\ra\infty}\frac{N(X)}{\#\{x\in\O_K:\|x\|\leq X\}} = \prod_{\p} (1-\meas(R_\p))
\]
as claimed.

Since $R_\p$ is a compact open proper subset of $\O_{K,\p}$, we have $\meas(R_\p)<1$ for all primes~$\p$. Moreover, $\meas(R_\p)\leq\meas(T+\p^2)\leq\#T\cdot\Nm(\p)^{-2}$. Since $\sum_\p\Nm(\p)^{-2}<\infty$, we conclude that $\prod_\p(1-\meas(R_\p))>0$.

\item
This follows immediately from (a).
\qedhere
\end{enumerate}
\end{proof}

Only the important task of proving the tail estimate remains:

\begin{proof}[Proof of \Cref{tail_estimate}]
We can assume without loss of generality that $\|\cdot\|$ is the supremum norm given by $\|x\|=\max_{\sigma:K\rightarrow\C}|\sigma(x)|$.

We have
\[
N'(X,M) \leq 1 + \sum_{\a: \Nm(\a)>M} \#\{0\neq x\in\a^k:\|x\|\leq X\}.
\]
The image of $\a^k$ under the Minkowski embedding $K\hookrightarrow\R^r\times\C^s$ is a lattice of rank $n$ with covolume proportional to $\Nm(\a^k)$. Denote the successive minima of this lattice with respect to $\|\cdot\|$ by $\lambda_1(\a^k),\dots,\lambda_n(\a^k)$. If $\lambda_1(\a^k) > X$, then we by definition have
\[
\#\{0\neq x\in\a^k:\|x\|\leq X\} = 0.
\]
For $\lambda_1(\a^k) \leq X$, we obtain the following classical upper bound. (For example, combine \cite[Lemma 2]{schmidt-lattice-counting} with Minkowski's second theorem.)
\begin{align*}
\#\{0\neq x\in\a^k:\|x\|\leq X\}
&\ll \sum_{i=0}^{n} \frac{X^i}{\lambda_1(\a^k)\cdots\lambda_i(\a^k)}
\leq \sum_{i=0}^{n} \frac{X^i}{\lambda_1(\a^k)^i}
\ll \frac{X^n}{\lambda_1(\a^k)^n}.
\end{align*}
Any non-zero element $x$ of $\a^k$ satisfies $\Nm(\a)^k\leq|\Nm_{K|\Q}(x)|=\prod_{\sigma:K\ra\C}|\sigma(x)|\leq\|x\|^n$. Therefore, $\lambda_1(\a^k)^n \geq \Nm(\a)^k$, so we obtain
\[
\#\{0\neq x\in\a^k:\|x\|\leq X\}
\ll \frac{X^n}{\Nm(\a)^k}.
\]
Hence,
\begin{align*}
N'(X,M)
&\ll 1 + \sum_{\a: \Nm(\a)>M}\frac{X^n}{\Nm(\a)^k} \ll 1 + \frac{X^n}{M^{k-1}}.
\qedhere
\end{align*}
\qedhere
\end{proof}

\section{Linear maps with local conditions}\label{section_thm_main}

\subsection{Statement}

Any $\Z$-linear map $A:\O_K\ra\O_L$ gives rise to a $\Z_p$-linear map $A:\O_{K,p}\ra\O_{L,p}$ for any prime number $p$.

The following theorem is the main result of this section:

\begin{theorem}\label{thm_main}
Let $R$ and $S$ be sieves for $K$ and $L$, respectively.
Let $A:\O_K\ra\O_L$ be a $\Z$-linear map such that there are infinitely many prime numbers $p$ satisfying the following conditions:
\begin{enumerate}[label=(\alph*)]
\item\label{thm_main_cond_inc} $A(V_p(K,R))\subseteq V_p(L,S)$.
\item\label{thm_main_cond_ntr} $A(\O_K)\cap S_\q\neq\emptyset$ for all primes $\q\mid p$ of $L$.
\item\label{thm_main_cond_vol} $\sum_{\p\mid p\textnormal{ prime of }K}\meas(R_\p)<1$.
\item\label{thm_main_cond_spl} $p$ splits completely in both $K$ and $L$.
\end{enumerate}
Then, $A$ is the composition $M_\varepsilon\circ\tau$ of a $\Q$-algebra homomorphism $\tau:K\ra L$ and the multiplication by $\varepsilon$ map $M_\varepsilon:\O_L\ra\O_L$ for some $\varepsilon\in\O_L$.
\end{theorem}

We first point out that there are always infinitely many prime numbers $p$ satisfying \ref{thm_main_cond_spl}:

\begin{lemma}\label{lemma_infinitely_many_split}
For any étale $\Q$-algebras $K,L$, there are infinitely many prime numbers $p$ splitting completely in $K$ and $L$.
\end{lemma}
\begin{proof}
Write $K=K_1\times\cdots\times K_r$ and $L=L_1\times\cdots\times L_s$ as products of number fields. Let $F$ be the compositum of the Galois closures of $K_1,\dots,K_r,L_1,\dots,L_s$. There are infinitely many prime numbers that split completely in $F$ (and hence in $K$ and in $L$), which can be seen either by the Chebotarev density theorem (see for example \cite[Corollary VII.13.6]{neukirch-algebraic-number-theory}), or by a much simpler purely algebraic argument (see \cite{poonen-infinitely-many-completely-split-primes}).
\end{proof}

\begin{remark}
None of the conditions \ref{thm_main_cond_inc}--\ref{thm_main_cond_spl} can be omitted. In fact, for \emph{all} $\Z$-linear bijections $A:\O_K\ra\O_L$, there are sieves $R$ and $S$ such that there are infinitely many prime numbers $p$ satisfying all of the conditions but one: (When omitting \ref{thm_main_cond_spl}, we assume that there are infinitely many prime numbers $p$ that are inert in $K$ or $L$.)
\begin{enumerate}[label=(\alph*)]
\item
Without \ref{thm_main_cond_inc}, any sieves $R,S$ with $S_\q\neq\emptyset$ work.
\item
Without \ref{thm_main_cond_ntr}, we could take $S_\q=\emptyset$ for all $\q$, so that $V_p(L,S)=\O_{L,p}$.
\item
Without \ref{thm_main_cond_vol}, we could take $R_\p=\O_{K,\p}$ for all $\p$, so that $V_p(K,R)=\emptyset$. Even assuming $R_\p\neq\O_{K,\p}$ is not enough: Taking any $\emptyset\neq S_\q\neq\O_{L,\q}$, the set $A^{-1}(\O_{L,p}\setminus V_p(L,S))\subseteq\O_{K,p}=\prod_{\p\mid p}\O_{K,\p}$ is non-empty and open. Hence, it contains a set of the form $\prod_{\p\mid p} C_\p$ with non-empty compact open sets $C_\p$. The sieve $R_\p=\O_{K,\p}\setminus C_\p$ then satisfies \ref{thm_main_cond_inc}.
\item
Let $S$ be the $k$-free sieve for any $k\geq1$. We can construct a sieve $R$ satisfying \ref{thm_main_cond_inc}--\ref{thm_main_cond_vol} for all but finitely many prime numbers~$p$ that are inert in $K$ or $L$ as follows. (This has to do with the fact that the set $V_p(K,R)$ can have more symmetries when $p$ is inert than when it splits completely. See also \Cref{fig_sqfree_sieve_nf}.)
\begin{itemize}
\item
For every prime number $p$ that is inert in $L$ and for every prime $\p\mid p$ of $K$, define $R_\p=\p^k$. For any $x\in V_p(K,R)$, we have $p^k\nmid x$ and therefore $p^k\nmid A(x)$, so $x\in V_p(L,S)$.
\item
For every prime number $p$ that is inert in $K$ but not in $L$, define $R_p=\bigcup_{\q\mid p}A^{-1}(\q^k)$. By definition, $A(V_p(K,R))\subseteq V_p(L,S)$. The measure $\meas(R_p)\leq\#\{\q\mid p\}\cdot\Nm(\q^k)$ is smaller than $\frac1n$ for all but finitely many prime numbers $p$, implying \ref{thm_main_cond_vol}.
\end{itemize}
It seems plausible, however, that the condition that $p$ split completely in $K$ and $L$ could be weakened, as long as certain other appropriate splitting behaviors occur infinitely often.
\end{enumerate}
\end{remark}

\subsection{Application to \texorpdfstring{$k$}{k}-free integers}

\Cref{thm_main} immediately proves the implication (b) $\Rightarrow$ (c) in \Cref{intro_symmetries,intro_linear_maps}:

\begin{corollary}\label{cor_linear_maps}
Let $K,L$ be étale $\Q$-algebras and let $k,l\geq1$.
\begin{enumerate}[label=(\alph*)]
\item If $A:\O_K\ra\O_L$  is a $\Z$-linear map with $A(V_{K,k,p})\subseteq V_{L,l,p}$ for all primes numbers $p$, then $A$ is the composition $M_\varepsilon\circ\tau$ of a $\Q$-algebra homomorphism $\tau:K\ra L$ and the multiplication by $\varepsilon$ map $M_\varepsilon:\O_L\ra\O_L$ for some $\varepsilon\in\O_L$.
\item If moreover $A$ is bijective, then $\tau$ is an isomorphism and $\varepsilon$ is a unit.
\end{enumerate}
\end{corollary}
\begin{proof}
\begin{enumerate}[label=(\alph*)]
\item Let $R$ be the $k$-free sieve on $K$ and let $S$ be the $l$-free sieve on $L$. The condition $A(V_{K,k,p})\subseteq V_{L,l,p}$ is equivalent to $A(V_p(K,R))\subseteq V_p(L,S)$. We have $0\in A(\O_K)\cap S_\q$ for all primes $\q$ of $L$. Moreover, $\meas(R_\p)=\Nm(\p)^{-k}$ is less than $\frac1n$ for all but finitely many primes $\p$. Hence, the claim follows from \Cref{thm_main} together with \Cref{lemma_infinitely_many_split}.
\item The image of $A = M_\varepsilon\circ\tau$ is contained in the ideal of $\O_L$ generated by $\varepsilon$. Since $A$ is surjective, $\varepsilon$ must be a unit. Since $A$ is bijective, $\tau$ must then be an isomorphism.
\qedhere
\end{enumerate}
\end{proof}

\subsection{Geometric lemma}

The following statement can be viewed as a geometric analogue of \Cref{intro_linear_maps}. We consider the $F$-algebra $F^n=F\times\cdots\times F$ and denote by $(F^\times)^n$ the set of invertible elements, i.e., the set of vectors in $F^n$ with non-zero coordinates.

\begin{lemma}
\label{geometric_lemma}
Let $n,m\geq0$, let $F$ be a field with at least $3$ elements, and let $A:F^n\ra F^m$ be an $F$-linear bijection. The following are equivalent:
\begin{enumerate}[label=(\alph*)]
\item The map $A$ satisfies $A((F^\times)^n) \subseteq (F^\times)^m$.
\item The map $A$ is the composition $M_\varepsilon\circ\tau$ of an $F$-algebra homomorphism $\tau:F^n\ra F^m$ and the coordinate-wise multiplication by $\varepsilon$ map for some $\varepsilon\in(F^\times)^m$.
\end{enumerate}
\end{lemma}
\begin{proof}
The implication (b) $\Rightarrow$ (a) is clear. Conversely, assume (a).

Let $(a_{ij})_{i,j}$ be the $m\times n$-matrix representing $A$.
By assumption, the $i$-th coordinate $\sum_j a_{ij}$ of $A(1,\dots,1)$ has to be non-zero for all $i$. In particular, every row of $A$ needs to contain at least one non-zero entry.

Assume a row of this matrix contains two non-zero entries $a_{ij_1}$ and $a_{ij_2}$. Consider the vectors $x=(x_j)_j\in F^n$ with $x_j=1$ for all $j\neq j_1,j_2$. For every choice of $x_{j_1}$, there is exactly one choice of $x_{j_2}$ such that the $i$-th coordinate $\sum_j a_{ij}x_j$ of $A(x)$ is zero, and for every choice of $x_{j_2}$, there is exactly one such choice of $x_{j_1}$. Out of these $\#F$ choices of pairs $(x_{j_1},x_{j_2})$, at most two have one of the coordinates equal to $0$. Since $\#F\geq3$, there is a choice such that the coordinates of $x$ are all non-zero. On the other hand, the $i$-th coordinate of $A(x)$ is zero by construction, contradicting the assumption that $A((F^\times)^n)\subseteq(F^\times)^m$.

Hence, every row of $A$ contains exactly one non-zero entry, say $a_{i\rho(i)}\neq0$ for $\rho(i)\in\{1,\dots,n\}$. Then, $A$ is the composition of the $F$-algebra homomorphism $\tau:F^n\ra F^m$ given by $(x_j)_j\mapsto(x_{\rho(i)})_i$ and the multiplication by $\varepsilon=(a_{i\rho(i)})_i$ map $M_\varepsilon:F^m\ra F^m$.
\end{proof}
\begin{remark}
The lemma can fail for $F=\F_2$: Take the map $A:\F_2^3\ra\F_2^3$ given by the matrix
\[
A =
\mat{
1&0&0\\
0&1&0\\
1&1&1
}.
\]
It sends $(\F_2^\times)^3=\{(1,1,1)\}$ to itself, but isn't of the form $M_\varepsilon\circ\tau$. 
\end{remark}

\subsection{Reduction to the geometric lemma}

Consider any étale $\Q$-algebra $K$ of degree $n$. For any field extension $F$ of $\Q$, we obtain an étale $F$-algebra $K\otimes_\Q F$ of degree $n$ with a natural inclusion $K\hookrightarrow K\otimes_\Q F$. We say that $K$ \emph{splits completely} over $F$ if the $F$-algebra $K\otimes_\Q F$ is isomorphic to $F^n = F\times\cdots\times F$.

\begin{lemma}\label{lemma_norm_implies_b}
Let $K,L$ be étale $\Q$-algebras of degrees $n,m$. Let $F$ be any field extension of $\Q$ over which $K$ and $L$ split completely. Let $A:K\ra L$ be a $\Q$-linear map, which we extend to an $F$-linear map $K\otimes F\ra L\otimes F$. Assume that $A((K\otimes F)^\times)\subseteq(L\otimes F)^\times$.

Then, $A$ is the composition $M_\varepsilon\circ\tau$ of a $\Q$-algebra homomorphism $\tau:K\ra L$ and the multiplication by $\varepsilon$ map $M_\varepsilon:L\ra L$ for some $\varepsilon\in L^\times$.
\end{lemma}

\begin{remark}
The assumption that $L$ splits completely over $F$ is crucial. If $L\otimes F$ is a field, for example, then \emph{every} injective $\Q$-linear map $A:K\ra L$ satisfies $A((K\otimes F)^\times)\subseteq(L\otimes F)^\times$.

The assumption that $K$ splits completely over $F$ could be omitted and is only included in the statement for the sake of simplicity.
\end{remark}

\begin{proof}
By assumption, $A(1)\in L^\times$. Replacing $A$ by $M_{A(1)}^{-1}\circ A$, we can assume that $A(1)=1$.

Via the isomorphisms $K\otimes F\cong F^n$ and $L\otimes F\cong F^m$, we can interpret the $F$-linear map $A:K\otimes F\ra L\otimes F$ as an $F$-linear map $A':F^n\ra F^m$. By \Cref{geometric_lemma}, we have $A'=M_\varepsilon\circ\tau$ for an $F$-algebra homomorphism $\tau:F^n\ra F^m$ and some $\varepsilon\in(F^\times)^m$. Since $A(1)=1$, we must have $\varepsilon=1$, so $A'=\tau$, which is a ring homomorphism. The corresponding map $A:K\otimes F\ra L\otimes F$ is then also a ring homomorphism. Therefore, its $\Q$-linear restriction $A:K\ra L$ is a $\Q$-algebra homomorphism.
\end{proof}

\subsection{Local lemma}

The following statement will be useful for handling sieves where (some) $R_\p$ contain a non-zero residue class.

\begin{lemma}\label{lemma_cover}
Let $M$ be a $\Z_p$-submodule of $\Z_p^n$ and let $R_1,\dots,R_n\subseteq\Z_p$ be compact open sets with $\sum_{i=1}^n\meas(R_i)<1$. If $M\cap(\Z_p^\times)^n\neq\emptyset$, then $(a+M)\cap\prod_{i=1}^n(\Z_p\setminus R_i)\neq\emptyset$ for all $a\in\Z_p^n$.
\end{lemma}
\begin{proof}
Let $x$ lie in $M\cap(\Z_p^\times)^n$. Now, the set $a+M$ contains the image of the map $f:\Z_p\ra\Z_p^n$ sending $t$ to $a+t\cdot x$. The composition $\pi_i\circ f:\Z_p\ra\Z_p$ with the $i$-th coordinate projection $\pi_i:\Z_p^n\ra\Z_p$ sends any $t$ to $a_i+t\cdot x_i$. The assumption $x_i\in\Z_p^\times$ means that this map $\pi_i\circ f$ is bijective and measure-preserving. (It simply permutes the residue classes modulo any $\p^k$.) Hence, $\sum_{i=1}^n\meas((\pi_i\circ f)^{-1}(R_i))=\sum_{i=1}^n\meas(R_i)<1=\meas(\Z_p)$, so there is an element $t\in\Z_p$ such that $\pi_i\circ f(t)\notin R_i$ for $i=1,\dots,n$. Then, $f(t)$ lies in $(a+M)\cap\prod_{i=1}^n(\Z_p\setminus R_i)$.
\end{proof}
\begin{figure}[h]
\begin{tikzpicture}[scale=3,decoration=brace]
\newcommand{\drawrone}[3]{
	\def\a{#1}
	\def\b{#2}
	\begin{scope}[on background layer]
		\fill[lightgray] (#1,0) rectangle (#2,1);
	\end{scope}
	\draw[gray] (#1,0) rectangle (#2,1);
	\draw[|-|] (#1,-0.05) -- node[inner sep=1] (T) {} (#2,-0.05);
	\draw[->] (R1) to[out=#3,in=-90] (T);
}
\newcommand{\drawrtwo}[3]{
	\begin{scope}[on background layer]
		\fill[lightgray] (0,#1) rectangle (1,#2);
	\end{scope}
	\draw[gray] (0,#1) rectangle (1,#2);
	\draw[|-|] (-0.05,#1) -- node[inner sep=1] (T) {} (-0.05,#2);
	\draw[->] (R2) to[out=#3,in=180] (T);
}
\node (R1) at (0.5,-0.25) {$R_1$};
\node (R2) at (-0.3,0.5) {$R_2$};
\drawrone{0}{0.2}{180}
\drawrone{0.4}{0.55}{100}
\drawrone{0.9}{1}{0}
\drawrtwo{0.25}{0.4}{-45}
\drawrtwo{0.5}{0.8}{45}
\draw (0,0) rectangle (1,1);
\draw[thick] (0,0.2) -- (1,0.7);
\draw[thick] (0,0.7) -- (0.6,1);
\draw[thick] (0.6,0) -- (1,0.2);
\node (aM) at (1.7,0.5) {$\im(f)\subseteq a+M$};
\draw[->] (aM) to[out=180,in=-20] (0.8,0.57);
\draw[->] (aM) to[out=185,in=60] (0.8,0.13);
\draw[->] (aM) to[out=175,in=-50] (0.5,0.93);
\end{tikzpicture}
\caption{Illustration for \Cref{lemma_cover}. The projection of $\im(f)\subseteq a+M$ onto either of the two coordinates is surjective.}
\end{figure}

\subsection{Proof}

We now show the main statement of this section, \Cref{thm_main}:

\begin{proof}[Proof of \Cref{thm_main}]
Let $n$ be the degree of the $\Q$-algebra $K$. Let $F$ be the smallest number field over which $K$ and $L$ split completely. (Writing $K=K_1\times\cdots\times K_r$ and $L=L_1\times\cdots\times L_s$ as products of number fields, $F$ is the compositum of the Galois closures of $K_1,\dots,K_r,L_1,\dots,L_s$.) Every prime number that splits completely in $K$ and $L$ also splits completely in $F$.

We extend $A$ to an $F$-linear map $A:K\otimes F\ra L\otimes F$. If $A((K\otimes F)^\times)\subseteq(L\otimes F)^\times$, then the claim follows from \Cref{lemma_norm_implies_b}. Assume that there is some $x\in(K\otimes F)^\times\cong(F^\times)^n$ with $A(x)\notin(L\otimes F)^\times$.

The coordinates $x_1,\dots,x_n\in F^\times$ of $x$ can have non-zero $\mathfrak P$-adic valuation for only finitely many primes $\mathfrak P$ of $F$. Discarding those, by the assumption of \Cref{thm_main}, there is still a prime number $p$ such that:
\begin{enumerate}[label=(\alph*)]
\item $A(V_p(K,R))\subseteq V_p(L,S)$.
\item\label{main_pf_cond_ntr} $A(\O_K)\cap S_\q\neq\emptyset$ for all primes $\q\mid p$ of $L$.
\item\label{main_pf_cond_vol}  $\sum_{\p\mid p\textnormal{ prime of }K}\meas(R_\p)<1$.
\item $p$ splits completely in both $K$ and $L$ (and therefore in $F$).
\item The coordinates $x_1,\dots,x_n$ have $\mathfrak P$-adic valuation zero for some prime $\mathfrak P\mid p$ of $F$.
\end{enumerate}

Since $p$ splits completely in $F$, we have $F_{\mathfrak P}\cong\Q_p$ and $\O_{F,\mathfrak P}\cong\Z_p$.
Now, we embed $x\in K\otimes F\cong F^n$ into $K\otimes F_{\mathfrak P}\cong K\otimes\Q_p\cong\prod_{\p\mid p}K_\p\cong\Q_p^n$. Since the coordinates have valuation zero, the resulting point in fact lies in the unit group of the ring of integers: $x \in (\O_K\otimes_\Z\O_{F,\mathfrak P})^\times\cong(\O_K\otimes_\Z\Z_p)^\times\cong\prod_{\p\mid p}\O_{K,\p}^\times\cong(\Z_p^\times)^n$.

On the other hand, $A(x)$ is not invertible in the finite-dimensional $F$-algebra $L\otimes F$. It is therefore a zero divisor, which of course stays a zero divisor in $L\otimes F_{\mathfrak P}\cong L\otimes\Q_p\cong\prod_{\q\mid p}L_\q$.
Hence, the projection $\pi_\q(A(x))$ of $A(x)\in\prod_{\q\mid p}L_\q$ onto at least one factor $L_\q$ must be zero. In summary: $x$ lies in the intersection of $(\O_K\otimes\Z_p)^\times\cong(\Z_p^\times)^n$ and the kernel $M\subseteq\O_K\otimes\Z_p\cong\Z_p^n$ of the map $\pi_\q\circ A:\O_K\otimes\Z_p\ra\O_{L,\q}$.

By \ref{main_pf_cond_ntr}, we can pick some $a\in\O_K\otimes\Z_p\cong\Z_p^n$ such that $b:=\pi_\q(A(a))\in S_\q$. Since $\sum_{\p\mid p}\meas(R_\p)<1$ by \ref{main_pf_cond_vol} and the intersection of $(\O_K\otimes\Z_p)^\times\cong(\Z_p^\times)^n$ and $M\subseteq\O_K\otimes\Z_p\cong\Z_p^n$ is non-empty, it follows by \Cref{lemma_cover} that $a+M = (\pi_\q\circ A)^{-1}(b)$ intersects $V_p(K,R)=\prod_{\p\mid p}\O_{K,\p}\setminus R_\p\subseteq\prod_{\p\mid p}\O_{K,\p}\cong\O_K\otimes\Z_p$. Let $y$ lie in the intersection. Then, $y\in V_p(K,R)$ but $\pi_\q\circ A(y)=b\in S_\q$ and therefore $A(y)\notin V_p(L,S)$. This contradicts the assumption that $A(V_p(K,R))\subseteq V_p(L,S)$.
\end{proof}

\section{Dynamical systems}\label{section_dynamical_systems}

\subsection{Basic definitions}

We will work in the category of topological dynamical systems, which we define as follows.

A \emph{topological dynamical system} $(\X,G)$ consists of a group $G$ and a compact topological space~$\X$ together with a (left) action of $G$ on $\X$ by homeomorphisms. We write the group operation and the action additively in this paper.

A \emph{morphism} (also called \emph{intertwiner}) $(f,A):(\X,G)\ra(\Y,H)$ between topological dynamical systems consists of a group homomorphism $A:G\ra H$ and a continuous map $f:\X\ra \Y$ such that $f(g+X)=A(g)+f(X)$ for all $g\in G$ and $X\in \X$:
\[
\begin{tikzcd}
\X \rar{g} \dar[swap]{f} & \X \dar{f} \\
\Y \rar[swap]{A(g)} & \Y
\end{tikzcd}
\]

A \emph{factor map} is a morphism $(f,A)$ for which both $A$ and $f$ are surjective. If there is a factor map $(f,A):(\X,G)\twoheadrightarrow(\Y,H)$, we say that $(\Y,H)$ is a \emph{factor system} of $(\X,G)$.

A morphism $(f,A)$ in the category of topological dynamical systems is an isomorphism if and only if $A$ is a group isomorphism and $f$ is a homeomorphism. Isomorphic topological dynamical systems are also called \emph{topologically conjugate}.

We call the group of automorphisms $(f,A):(\X,G)\stackrel\sim\ra(\X,G)$ of a topological dynamical system $(\X,G)$ its \emph{extended symmetry group} and denote it by $\ExSym(\X,G)$. (Extended symmetry groups were introduced and motivated in \cite{baake-roberts-yassawi}.)

This is not to be confused with its \emph{symmetry group} $\Sym(\X,G)$, which is the group of homeomorphisms $f:\X\ra \X$ such that $f(g+X)=g+f(X)$ for all $g\in G$ and $X\in \X$. (Many authors in symbolic dynamics of shift spaces call this the ``automorphism group''; see for example \cite[Chapter~3]{kitchens-symbolic-dynamics}.)
Note that $\Sym(\X,G)$ is the kernel of the group homomorphism
\begin{equation}\label{forgetful_map}
\ExSym(\X,G) \ra \Aut_{\textnormal{group}}(G),\quad (f,A) \mapsto A,
\end{equation}

\begin{remark}
\begin{enumerate}[label=(\alph*)]
\item
If $G$ acts faithfully on $\X$, then the extended symmetry group is the normalizer of $G$ in the homeomorphism group of~$\X$.
\item
The symmetry group is the centralizer of $G$ in the homeomorphism group of $\X$.
\end{enumerate}
\end{remark}

\begin{remark}
The symmetry group always contains the translation maps $f(X)=g+X$ by elements $g$ of the center of $G$. In general, there can be additional symmetries.
\end{remark}

\subsection{Shift spaces}

For any group $G$, endow the family $\P(G)$ of subsets of $G$ with the product topology under the natural identification of $\P(G)$ with $\{0,1\}^G$. The group $G$ acts continuously by translation on $\P(G)$.

For the purposes of this paper, a \emph{shift space} is a topological dynamical system $(\X,G)$ obtained from a closed translation-invariant subset $\X$ of $\P(G)$ with the translation action of $G$. The shift space is called \emph{hereditary} if for all $X \subseteq X' \in \X$, we have $X\in\X$. (See \cite{lind-marcus-symbolic-dynamics} or \cite{kitchens-symbolic-dynamics} for an introduction to one-dimensional shift spaces. Most of our systems are higher-dimensional, where it is
well known that important and difficult new phenomena
occur; see \cite{schmidt-algebraic-dynamical-systems} for background.)

The Curtis--Hedlund--Lyndon theorem allows us to concretely describe morphisms between shift spaces as block codes:

\begin{theorem}[{see \cite{chl}, \cite[Theorem~6.2.9]{lind-marcus-symbolic-dynamics} for the one-dimensional case}]\label{general_chl}
Let $(\X,G)$ and $(\Y,H)$ be shift spaces and let $(f,A):(\X,G)\ra(\Y,H)$ be a morphism. Then, there is a finite \emph{window} $M\subseteq G$ and a family $\T$ of \emph{patterns} $T\subseteq M$ such that for all $x\in G$ and $X\in\X$, we have $A(x)\in f(X)$ if and only if $X\cap(x+M)=x+T$ for some pattern $T\in\T$.
\end{theorem}
\begin{proof}
Since $\P(G)\cong\{0,1\}^G$ is compact, so are the closed open subfamilies
\[
\{X\in\X:0\in f(X)\}\quad\textnormal{and}\quad\{X\in\X:0\notin f(X)\}.
\]
By the definition of the product topology, there is then a finite set $M\subseteq G$ such that for all $X\in\X$, we can determine whether $0$ lies in the subset $f(X)$ of $H$ just from the intersection $X\cap M\subseteq M$. More precisely, there is a family $\T$ of subsets of $M$ such that for all $X\in\X$, we have $0\in f(X)$ if and only if the set $X\cap M$ belongs to $\T$.

By the intertwining property, $f(-x+X)=-A(x)+f(X)$ for all $x\in G$. It follows that $A(x)\in f(X)$ is equivalent to $0\in f(-x+X)$, which is equivalent to $(-x+X)\cap M\in\T$, which is equivalent to $X\cap (x+M)=x+T$ for some $T\in\T$.
\end{proof}

With a few (rather weak) assumptions on the shift spaces and the morphism, we can show the following basic properties, which except for \ref{special_chl_kernel} are implicit for example in \cite{positive-entropy-shifts}:

\begin{lemma}\label{special_chl}
Let $(\X,G)$ and $(\Y,H)$ be shift spaces and let $(f,A):(\X,G)\ra(\Y,H)$ be a morphism. Assume that $\X$ is hereditary, that $H\notin\Y$ (if $\Y$ is hereditary, this is equivalent to $\Y\neq\P(H)$, i.e., to $\Y$ being not the full shift space), that $A$ is surjective, and that the image of $f$ contains a non-empty set. Then:
\begin{enumerate}[label=(\alph*)]
\item\label{special_chl_block}
There is a finite window $M\subseteq G$ and a \emph{non-empty} family $\T$ of \emph{non-empty} patterns $T\subseteq M$ \emph{contained in $\X$} such that for all $x\in G$ and $X\in\X$, we have $A(x)\in f(X)$ if and only if $X\cap(x+M)=x+T$ for some pattern $T\in\T$.
\item\label{special_chl_kernel}
The group homomorphism $A:G\ra H$ has a finite kernel.
\item\label{special_chl_singleton}
If $f:\X\ra\Y$ is injective, then any family $\T$ as in (a) contains a singleton set.
\end{enumerate}
\end{lemma}
\begin{proof}
\begin{enumerate}[label=(\alph*)]
\item
Let $M$ and $\T$ as in \Cref{general_chl}. Since $\X$ is translation-invariant and hereditary, the condition $X\cap(x+M)=x+T$ with $X\in\X$ can only hold for $T\in\X$. Hence, we can simply remove from $\T$ all subsets not in $\X$ without affecting the validity of \Cref{general_chl}.

As the image of $f$ contains a non-empty set, $\X$ must be non-empty. Since $\X$ is hereditary, we must in fact have $\emptyset\in\X$.

If the family $\T$ was empty, then we would have $f(X)=\emptyset$ for all $X\in\X$, contradicting the assumption that the image of $f$ contains a non-empty set: For any $y\in H$, we can choose a preimage $x\in G$ under the map $A$. Then, $y\in f(X)$ is equivalent to $X\cap(x+M)=x+T$ for some pattern $T\in\T$, which is impossible if $\T$ is empty.

Similarly, if the family $\T$ contained the empty set, then we would have $f(\emptyset)=H$, which is impossible as $H\notin\Y$.

\item
Let $X\in\X$ with $f(X)\neq\emptyset$. Pick any $y\in f(X)$. Assume that the kernel of $f$ is infinite. We can then pick preimages $x_1,x_2\in G$ of $y$ under the map $A$ such that $x_1+M$ is disjoint from $x_2+M$.
Let $X' = X\cap(x_1+M)$. On the one hand, $X\cap(x_1+M)=X'\cap(x_1+M)$. Since $y=A(x_1)\in f(X)$, we see from \ref{special_chl_block} that $y=A(x_1)\in f(X')$. On the other hand, $X'\cap(x_2+M)=\emptyset$. Since $\emptyset\notin\T$, we obtain from \ref{special_chl_block} the contradiction $y=A(x_2)\notin f(X')$.
\item
By \ref{special_chl_block}, the subfamily $\T$ of $\X$ contains a non-empty set. Since the family $\X$ is hereditary, $\X$ must therefore contain a singleton set $\{x\}$. If $\T$ did not contain a singleton set, since it also does not contain the empty set, \ref{special_chl_block} would imply that $f(\emptyset)=\emptyset=f(\{x\})$, contradicting the injectivity of $f$.
\qedhere
\end{enumerate}
\end{proof}

\subsection{Admissible sets}

Any sieve $R$ for an étale $\Q$-algebra $K$ gives rise to a shift space in the following way:

\begin{definition}
We let $\X(K,R)$ be the family of subsets $X$ of $\O_K$ that are disjoint from some \emph{translate} $\delta_\p+R_\p$ of $R_\p$ for each prime $\p$ (with $\delta_\p\in\O_{K,\p}$). The elements of $\X(K,R)$ are called the \emph{admissible} subsets of $\O_K$ for $R$.

The family $\X(K,R)$ is a closed translation-invariant subset of $\P(\O_K)$, so we obtain a shift space $(\X(K,R),\O_K)$, which we denote by $\D(K,R)$ and call the \emph{admissible shift space associated to $R$}. (The claim that $\X(K,R)$ is closed follows from the fact that there are only finitely many distinct translates of the compact open set $R_\p$.)
\end{definition}

\begin{example}
If $R$ is the $k$-free sieve, then $\X(K,R)$ is the family of subsets of $\O_K$ that are disjoint from at least one residue class modulo $\p^k$ for each prime $\p$. This matches the local definition of $\X_{K,k}$ given in \Cref{introduction} (for $k\geq2$). We obtain the shift space $\D_{K,k}=\D(K,R)$ as in \Cref{introduction} and extend this definition to $k=1$.
\end{example}

\begin{remark}~
\begin{enumerate}[label=(\alph*)]
\item The space $\X(K,R)$ is hereditary, i.e., all subsets of admissible sets are admissible.
\item If $R_\p\neq\O_{K,\p}$ for all primes $\p$, then $\{x\}\in\X(K,R)$ for all $x\in\O_K$.
\item If $R_\p\neq\emptyset$ for some prime $\p$, then $\O_K\notin\X(K,R)$.
\item A subset $T$ of $\O_K$ is admissible if and only if $-T+R_\p\neq\O_{K,\p}$ for all primes $\p$, where we use the notation $-T+R_\p = \{-t+r:t\in T,r\in R_\p\}$.
\item If $T_1\subseteq T_2\subseteq\dots$ are admissible subsets of $\O_K$, then so is their union $\bigcup_{i\geq1} T_i$. In particular, by Zorn's lemma, there is a maximal admissible subset of $\O_K$ (with respect to inclusion).
\end{enumerate}
\end{remark}

We will focus on sieves satisfying the following two properties:

\begin{definition}
A sieve $R$ is \emph{non-large}\footnote{The term \emph{non-large} is inspired by the \emph{large sieve} from analytic number theory. (See \cite[Chapter 12]{serre-lectures-mordell-weil}.)} if $R_\p\neq\O_{K,\p}$ for all primes $\p$ and if for every $\varepsilon>0$, there are only finitely many primes $\p$ with $\meas(R_\p)>\varepsilon$.

A sieve $R$ is \emph{cofinite} if there are only finitely many primes $\p$ with $R_\p=\emptyset$.
\end{definition}

\begin{example}
For any $k\geq1$, the $k$-free sieve is non-large and cofinite as there are only finitely many primes with bounded norm.
\end{example}

The following consequence of the Chinese remainder theorem allows us to construct finite admissible subsets of $\O_K$ that intersect ``many'' residue classes and whose elements are arbitrarily far apart. (On first reading, we suggest taking $T_1=\cdots=T_r=\{0\}$.)

\begin{lemma}\label{dynamical_crt2}
Let $R$ be a non-large sieve for $K$. Let $M$ be a finite subset of $\O_K$, let $T_1,\dots,T_r\in\X(K,R)$ be finite (admissible) subsets of $\O_K$, and let $U_1,\dots,U_r$ be open subsets of $\prod_\p\O_{K,\p}$ such that $U_i$ intersects $\prod_\p\O_{K,\p}\setminus(-T_i+R_\p)$ for all $i=1,\dots,r$. Then, there are elements $x_1,\dots,x_r$ of $\O_K$ such that:
\begin{enumerate}[label=(\alph*)]
\item\label{dyncrt2a} For $i=1,\dots,r$, the image of $x_i$ in $\prod_\p\O_{K,\p}$ lies in $U_i$.
\item\label{dyncrt2b} The set $X:=\bigcup_{i=1}^r(x_i+T_i)$ is admissible: $X \in \X(K,R)$.
\item\label{dyncrt2c} The sets $x_1+M,\dots,x_r+M$ are pairwise disjoint.
\end{enumerate}
\end{lemma}
\begin{remark}
If $U_i$ intersects $\prod_\p\O_{K,\p}\setminus(-T_i+R_\p)$, we must in particular have $-T_i+R_\p\neq\O_{K,\p}$ for all primes $\p$. This is equivalent to the admissibility of $T_i$.
\end{remark}
\begin{proof}
Shrinking the sets $U_1,\dots,U_r$, we can assume without loss of generality that they are of the form
\[
U_i=\prod_{\p\in\Omega}C_{i,\p}\times\prod_{\p\notin\Omega}\O_{K,\p}
\]
with a finite set $\Omega$ of primes and with non-empty (compact) open subsets $C_{i,\p}$ of $\O_{K,\p}\setminus(-T_i+R_\p)$ for $\p\in\Omega$. Shrinking the sets further, we can moreover assume that $\Omega$ contains the (finitely many!)\ primes $\p$ with $\sum_{i=1}^r|T_i|\cdot\meas(R_\p)\geq1$.

By the Chinese remainder theorem, for each $i$, there are infinitely many $x_i\in\O_K$ with $x_i \in C_{i,\p}$ for all $\p\in\Omega$.
We can therefore pick $x_1,\dots,x_r$ satisfying these conditions and arbitrarily far apart, so that the sets $x_1+M,\dots,x_r+M$ are pairwise disjoint. By construction, the image of $x_i$ in $\prod_\p\O_{K,\p}$ lies in $U_i$.

It remains to check that the set $X=\bigcup_{i=1}^r(x_i+T_i)$ is admissible, i.e., that it is disjoint from some translate of $R_\p$ for all $\p$.
\begin{itemize}
\item
For $\p\in\Omega$, we have $X\cap R_\p=\emptyset$: As $x_i$ lies in the set $C_{i,\p}$, which is disjoint from $-T_i+R_\p$, we have $(x_i+T_i)\cap R_\p=\emptyset$.
\item
For $\p\notin\Omega$, since $\sum_{i=1}^r|T_i|\cdot\meas(R_\p)<1$, the $\sum_{i=1}^r|T_i|$ translates $-x_i-t+R_\p$ of $R_\p$ for $i=1,\dots,r$ and $t\in T_i$ cannot cover $\O_{K,\p}$, so we can pick some $\delta_\p\in\O_{K,\p}\setminus\bigcup_{i=1}^r(-x_i-T_i+R_\p)$. Then, $X\cap(-\delta_\p+R_\p)=\emptyset$.
\qedhere
\end{itemize}
\end{proof}

As a consequence, we obtain the following lemma. (This lemma will not be used later. However, the proof seems like a good exercise on how to apply \Cref{dynamical_crt2}, preparing us for slightly more involved applications in \Cref{dynmor}.)

\begin{lemma}\label{lemma_admissible_inclusions}
Let $R$ and $S$ be non-large sieves for $K$.
\begin{enumerate}[label=(\alph*)]
\item We have $\X(K,R)\subseteq\X(K,S)$ if and only if for each prime $\p$, the set $S_\p$ is a subset of some translate $\delta_\p+R_\p$ of $R_\p$.
\item We have $\X(K,R)=\X(K,S)$ if and only if for each prime $\p$, the set $S_\p$ is some translate $\delta_\p+R_\p$ of $R_\p$.
\end{enumerate}
\end{lemma}
\begin{proof}
\begin{enumerate}[label=(\alph*)]
\item ``$\Leftarrow$'' is clear. To prove ``$\Rightarrow$'', assume $\X(K,R)\subseteq\X(K,S)$. Let $R_\p$ and $S_\p$ be defined modulo $\p^k$. Let $C_1,\dots,C_r\subseteq\O_{K,\p}$ be the residue classes modulo $\p^k$ that are disjoint from $R_\p$, so that $\O_{K,\p}\setminus R_\p=C_1\cup\cdots\cup C_r$. Apply \Cref{dynamical_crt2} to the sieve $R$, the sets $M=\emptyset$ and $T_1=\cdots=T_r=\{0\}$, and the open subsets $U_1,\dots,U_r$ of $\prod_{\p'}\O_{K,\p'}$ with $U_i=C_i\times\prod_{\p'\neq\p}\O_{K,\p'}$. We obtain a set $X\in\X(K,R)\subseteq\X(K,S)$ containing representatives $x_1,\dots,x_r$ of the residue classes $C_1,\dots,C_r$. Since $X$ lies in $\X(K,S)$, it must be disjoint from some translate $-\delta_\p+S_\p$ of $S_\p$. As $S_\p$ is defined modulo $\p^k$, the entire residue classes $C_1,\dots,C_r$ are disjoint from $-\delta_\p+S_\p$, so $\O_{K,\p}\setminus R_\p$ is disjoint from $-\delta_\p+S_\p$. This implies that $S_\p$ is a subset of $\delta_\p+R_\p$.
\item ``$\Leftarrow$'' is clear. To prove ``$\Rightarrow$'', we apply (a) to both inclusions. We learn that $S_\p\subseteq\delta_\p+R_\p\subseteq\delta_\p+\delta'_\p+S_\p$ for some $\delta_\p,\delta'_\p\in\O_{K,\p}$. Since $S_\p$ and its translate $\delta_\p+\delta'_\p+S_\p$ consist of the same number of residue classes modulo $\p^k$, we must have equality.
\qedhere
\end{enumerate}
\end{proof}

\subsection{Criteria for the existence of morphisms}\label{section_shift_morphisms}

In this section, we study morphisms between the admissible shift spaces defined in the previous section. The main result is the following theorem, which gives necessary conditions for the existence of a morphism / factor map / isomorphism $(f,A):\D(K,R)\ra\D(L,S)$.

\begin{theorem}\label{dynmor}
Let $R$ and $S$ be non-large cofinite sieves for $K$ and $L$, respectively.
Let $(f,A):\D(K,R)\ra\D(L,S)$ be a morphism. Assume that $A$ is surjective and that the image of $f$ contains a non-empty set. Let $M$ and $\T$ as in \Cref{special_chl}\ref{special_chl_block}. Then:
\begin{enumerate}[label=(\alph*)]
\item\label{dynmora}
The group homomorphism $A$ is the composition of a $\Q$-algebra isomorphism $\tau:K\ra L$ and the multiplication by $\varepsilon$ map $M_\varepsilon$ for some unit $\varepsilon\in\O_L^\times$.

In particular, $K$ is isomorphic to $L$ and $A$ is bijective.

Note that this implies that the bijection $A:\O_K\ra\O_L$ extends to a bijection $A=M_\varepsilon\circ\tau:\O_{K,\p}\ra\O_{L,\tau(\p)}$ for all primes $\p$ of $K$.
\item\label{dynmorq}
For all primes $\p$ of $K$, the set $S_{\tau(\p)}\subseteq\O_{L,\tau(\p)}$ is a subset of some translate of $A(R'_\p)\subseteq\O_{L,\tau(\p)}$, where we define $R'_\p = \bigcap_{T\in\T}(-T+R_\p) \subseteq \O_{K,\p}$.
\item\label{dynmorc}
If $(f,A)$ is an isomorphism, then for all primes $\p$ of $K$, the set $S_{\tau(\p)}$ is a translate of $A(R_\p)$.
\item\label{dynmorb}
If $(f,A)$ is a factor map, then there is a set $T\in\T$ such that for all but finitely many primes $\p$ of $K$, the set $S_{\tau(\p)}$ is a translate of $A(-T+R_\p)$.
\end{enumerate}
\end{theorem}

\begin{remark}
When $K=L$ and $\X(K,R)\subseteq\X(L,S)$ and $f:\X(K,R)\ra\X(L,S)$ is the inclusion map, one can easily show $A=\id$ and the theorem is equivalent to \Cref{lemma_admissible_inclusions}.
\end{remark}

\begin{remark}
There is the following ``converse'' of \ref{dynmora} together with \ref{dynmorq}: Let $R$ and $S$ be sieves for $K$ and $L$, respectively. Let $A=M_\varepsilon\circ\tau$ as in \ref{dynmora}. Let $M$ be a finite subset of $\O_K$ and let $\T\subseteq\X(K,R)$ be a family of subsets of $M$. Assume that for each prime $\p$ of $K$, the set $S_{\tau(\p)}$ is a subset of some translate of $R'_\p=\bigcap_{T\in\T}(-T+R_\p)$. Then, we obtain a well-defined morphism $(f,A):\D(K,R)\ra\D(L,S)$ with the given window $M$ and pattern set $\T$, i.e., so that as in \Cref{special_chl}\ref{special_chl_block}, we have $A(x)\in f(X)$ if and only if $X\cap(x+M)=x+T$ for some pattern $T\in\T$.
\end{remark}

\begin{remark}
There is the following ``converse'' of \ref{dynmora} together with \ref{dynmorc}: Let $R$ and $S$ be sieves for $K$ and $L$, respectively. Let $A=M_\varepsilon\circ\tau$ as in \ref{dynmora}. Assume that for each prime $\p$ of $K$, the set $S_{\tau(\p)}$ is a translate of $A(R_\p)$. Then, there is an isomorphism $(f,A):\D(K,R)\ra\D(L,S)$ with $f(X) = A(X)$ for $X\in\X(K,R)$.
\end{remark}

\begin{remark}
In the proofs of \ref{dynmorq}--\ref{dynmorb}, we could assume without loss of generality that $K=L$ and $A=\id$ by composing with the isomorphism $(f',A^{-1}):\D(L,S)\ra\D(K,\widetilde S)$, where $f'(X)=A^{-1}(X)$ and $\widetilde S_\p = A^{-1}(S_{\tau(\p)})$.
\end{remark}

Our proof of \Cref{dynmor} borrows some ideas from the proof of \cite[Theorem 4.1]{positive-entropy-shifts}. The diagonalization argument in \ref{dynmorb} seems to be new.

\begin{proof}
First, note that the morphism $(f,A)$ satisfies the assumptions of \Cref{special_chl}.

\begin{enumerate}[label=(\alph*)]
\item
Since $\O_K\cong\Z^n$ has no non-trivial finite subgroups, \Cref{special_chl}\ref{special_chl_kernel} implies that the surjective map $A:\O_K\ra\O_L$ is also injective. Define the sieve $R'$ as in \ref{dynmorq}. Recall that the $\Z$-linear bijection $A:\O_K\ra\O_L$ gives rise to a $\Z_p$-linear bijection $A:\O_{K,p}\ra\O_{L,p}$ for each prime number $p$.

Using that $f:\X(K,R)\ra\X(L,S)$ maps admissible subsets of $\O_K$ for $R$ to admissible subsets of $\O_L$ for $S$, we now show:

\begin{claim}\label{claim1}
Let $T\in\T$ and let $p$ be any prime number. Then, $A(V_p(K,-T+R))\subseteq V_p(L,\delta+S)$ for some $\delta=(\delta_\q)_\q\in\prod_{\q}\O_{L,\q}$. (Here $-T+R$ refers to the sieve $R''$ defined by $R''_\p=-T+R_\p$.)
\end{claim}
\begin{proof}
Let $k\geq0$ be large enough so that all $R_\p$ for $\p\mid p$ and $S_\q$ for $\q\mid p$ are defined modulo $p^k$. The set $V_p(K,-T+R)\subseteq\O_{K,p}$ is a union of residue classes $C_1,\dots,C_r\subseteq\O_{K,p}$ modulo $p^k$.
Each residue class $C_i$ is by definition disjoint from $-T+R_\p$ for all primes~$\p\mid p$ of~$K$.
Apply \Cref{dynamical_crt2} to the sieve $R$, the finite set $M$ from \Cref{special_chl}\ref{special_chl_block}, the admissible sets $T_1=\dots=T_r=T$, and the open subsets $U_1,\dots,U_r$ with $U_i = C_i \times \prod_{\p\nmid p}\O_{K,\p}$. We obtain an admissible set $X=\bigcup_{i=1}^r(x_i+T)\in\X(K,R)$.

Applying $f$, we obtain an admissible set $f(X) \in \X(L,S)$. This set $f(X)$ must be disjoint from some translate $\delta_\q+S_\q$ of each $S_\q$.

Since $T\subseteq M$ and the sets $x_i+M$ for $i=1,\dots,r$ are pairwise disjoint, we have $X\cap(x_i+M)=x_i+T$ for all $i$. By \Cref{special_chl}\ref{special_chl_block}, this means that $A(x_1),\dots,A(x_r)$ lie in $f(X)$.

We conclude that $\{A(x_1),\dots,A(x_r)\}$ is disjoint from some translate $\delta_\q+S_\q$ of each $S_\q$ with $\q\mid p$. The set $S_\q$ is defined modulo $p^k$, so in fact the union $A(C_1)\cup\dots\cup A(C_r) = A(V_p(K,-T+R))$ of the residue classes modulo $p^k$ containing the numbers $A(x_1),\dots,A(x_r)$ must be disjoint from the translate $\delta_\q+S_\q$. In other words, $A(V_p(K,-T+R)) \subseteq V_p(L,\delta+S)$.
\end{proof}

We will apply \Cref{thm_main} to the sieves $-T+R$ and $\delta+S$ for $K$ and $L$, respectively, and to the $\Z$-linear bijection $A:\O_K\ra\O_L$. To do this, we need to check that there are infinitely many prime numbers satisfying conditions \ref{thm_main_cond_inc}--\ref{thm_main_cond_spl} of \Cref{thm_main}:
\begin{enumerate}[label=(\alph*)]
\item We have $A(V_p(K,-T+R))\subseteq V_p(L,\delta+S)$ for all prime numbers $p$ by \Cref{claim1}.
\item We have $A(\O_K)\cap(\delta_q+S_\q)=\delta_\q+S_\q\neq\emptyset$ for all but finitely many primes $\q$ of $L$ as $A$ is surjective and $S$ is cofinite.
\item We have $\meas(-T+R_\p)<\frac1n$ for all but finitely many primes $\p$ of $K$: Indeed, $\meas(-T+R_\p)\leq|T|\cdot\meas(R_\p)$ is less than $\frac1n$ for all but finitely many primes $\p$ as $R$ is non-large.
\item There are infinitely many prime numbers $p$ that split completely in $K$ and $L$ by \Cref{lemma_infinitely_many_split}.
\end{enumerate}
Hence, we can apply \Cref{thm_main}. We learn that $A$ is the composition of a $\Q$-algebra homomorphism $\tau:K\ra L$ and the multiplication by $\varepsilon$ map for some $\varepsilon\in\O_L$. Since $A:\O_K\ra\O_L$ is bijective, the map $\tau$ must be an isomorphism and we must have $\varepsilon\in\O_L^\times$. (The image of $A$ is contained in the ideal generated by $\varepsilon$.)

\item
From \ref{dynmora}, we have an isomorphism $A:\O_{K,\p}\ra\O_{L,\tau(\p)}$ for every prime $\p$ of $K$. With this additional knowledge, we can essentially repeat the proof of \Cref{claim1} to show the (slightly stronger) claim \ref{dynmorq}:

Let $k\geq0$ be large enough so $R_\p$ and $S_{\tau(\p)}$ are defined modulo $\p^k$ and $\tau(\p)^k$, respectively. The set $\O_{K,\p}\setminus R_\p'$ is a union of residue classes $C_1,\dots,C_r$ modulo $\p^k$. By the definition of $R'_\p$, for each residue class $C_i$, there is a pattern $T_i\in\T$ such that $C_i$ is not contained in $-T_i+R_\p$. Apply \Cref{dynamical_crt2} to the sieve $R$, the finite set $M$ from \Cref{special_chl}\ref{special_chl_block}, the admissible sets $T_1,\dots,T_r$, and the open subsets $U_1,\dots,U_r$ with $U_i = C_i \times \prod_{\p'\neq\p}\O_{K,\p'}$. We obtain a set $X=\bigcup_{i=1}^r(x_i+T_i)\in\X(K,R)$. As in the proof of \Cref{claim1}, it follows that $A(C_1)\cup\cdots\cup A(C_r)=A(\O_K\setminus R'_\p)=\O_L\setminus A(R'_\p)$ is disjoint from some translate of $S_{\tau(\p)}$. Hence, $S_{\tau(\p)}$ is a subset of some translate of $R'_\p$.

\item
We assume that $(f,A)$ is an isomorphism. By \Cref{special_chl}\ref{special_chl_singleton}, the family $\T$ contains a singleton set $T=\{t\}$. With \ref{dynmorq}, we conclude that for all primes $\p$, the set $S_{\tau(\p)}$ is a subset of some translate of $A(R_\p)$. Applying the same argument to the inverse morphism, we see that $A(R_\p)$ is a subset of some translate of $S_{\tau(\p)}$. Together, it follows that $S_{\tau(\p)}$ is a translate of $A(R_\p)$.

\item
We assume that $(f,A)$ is a factor map, so $f:\X(K,R)\ra\X(L,S)$ is surjective. Using that every admissible subset of $\O_L$ for $S$ has a preimage which is an admissible subset of $\O_K$ for $R$, we prove:

\begin{claim}\label{claim2}
Consider a finite set $\Omega$ of primes $\p$ of $K$. There is a collection of elements $\delta_\p\in\O_{K,\p}$ for $\p\in\Omega$ such that the following holds: For all $y=(y_\p)_\p \in \prod_{\p\in\Omega}\O_{L,\tau(\p)}\setminus S_{\tau(\p)}$, there is a set $T_y\in\T$ such that for all $\p\in\Omega$, the number $y_\p$ does not lie in $A(\delta_\p-T_y+R_\p)$.
\end{claim}
\begin{proof}
Let $k\geq0$ be large enough so that $R_\p$ is defined modulo $\p^k$ and $S_{\tau(\p)}$ is defined modulo $\tau(\p)^k$ for all $\p\in\Omega$. For each $\p\in\Omega$, write $\O_{L,\tau(\p)}\setminus S_{\tau(\p)}$ as a union of residue classes $C_{\p,i}$ modulo $\tau(\p)^k$ with $i=1,\dots,I_\p$.

Apply \Cref{dynamical_crt2} to the sieve $S$ for $L$, the sets $M=\emptyset$ and $T_i=\{0\}$ and to the open subsets $U_i = \prod_{\p\in\Omega}C_{\p,i_\p}\times\prod_{\p\notin\Omega}\O_{L,\tau(\p)}$ with $i=(i_\p)_\p\in\prod_{\p\in\Omega}\{1,\dots,I_\p\}$. We obtain an admissible set $Y\in\X(L,S)$ intersecting all sets $U_i$.
By the surjectivity of $f:\X(K,R)\ra\X(L,S)$, we can pick some preimage $X\in\X(K,R)$ of $Y$. Since $X$ is admissible for $R$, it is disjoint from some translate $\delta_\p+R_\p$ of $R_\p$ for each $\p\in\Omega$.

Now, take any $y=(y_\p)_\p\in\prod_{\p\in\Omega}\O_{L,\tau(\p)}\setminus S_{\tau(\p)}$. Take $i=(i_\p)_\p\in\prod_{\p\in\Omega}\{1,\dots,I_\p\}$ so that $C_{\p,i_\p}$ is the residue class containing $y_\p$ for all $\p\in\Omega$. Take an element $y'$ of $Y = f(X)$ whose image in $\prod_\q\O_{L,\q}$ lies in $U_i$. As $y'$ lies in $f(X)$, by \Cref{special_chl}\ref{special_chl_block}, we have $X\cap(A^{-1}(y')+M)=A^{-1}(y')+T_y$ for some pattern $T_y\in\T$.
Then, $A^{-1}(y')+T_y\subseteq X$ must be disjoint from $\delta_\p+R_\p$. Since $y'\equiv y_\p\mod\tau(\p)^k$ and $R_\p$ is a union of conjugacy classes modulo $\p^k$, it follows that $A^{-1}(y_\p)+T_y$ is disjoint from $\delta_\p+R_\p$ as well. Applying the ($\Z$-linear) bijection $A$, we conclude that $y_\p+A(T_y)$ is disjoint from $A(\delta_\p+R_\p)$, so $y_\p\notin A(\delta_\p-T_y+R_\p)$.
\end{proof}

Assume \ref{dynmorb} does not hold, so for every $T\in\T$, there are infinitely many primes $\p$ of $K$ such that $S_{\tau(\p)}$ is not a translate of $A(-T+R_\p)$. We will obtain a contradiction using a diagonalization argument. Denote the elements of $\T$ by $T_1,\dots,T_r$. By assumption, there are distinct primes $\p_1,\dots,\p_r$ of $K$ such that for each $1\leq i\leq r$, the set $S_{\tau(\p_i)}$ is not a translate of $A(-T_i+R_{\p_i})$.

From \Cref{claim2} with $\Omega=\{\p_1,\dots,\p_r\}$ we obtain numbers $\delta_1,\dots,\delta_r$ such that for every $y=(y_i)_i\in\prod_{i=1}^r\O_{L,\tau(\p_i)}\setminus S_{\tau(\p_i)}$, there is a set $T_y\in\T$ such that for all $1\leq i\leq r$, we have $y_i\notin A(\delta_i-T_y+R_{\p_i})$. Choosing $i$ so that $T_i=T_y$, we see that for every $y\in\prod_{i=1}^r\O_{L,\tau(\p_i)}\setminus S_{\tau(\p_i)}$, there is some $1\leq i\leq r$ such that $y_i\notin A(\delta_i-T_i+R_{\p_i})$. This simply means that there is some $1\leq i\leq r$ such that for all $y\in\O_{L,\tau(\p_i)}\setminus S_{\tau(\p_i)}$, we have $y\notin A(\delta_i-T_i+R_{\p_i})$. Then, $A(\delta_i-T_i+R_{\p_i})$ is a subset of $S_{\tau(\p_i)}$. Conversely, according to \ref{dynmorq}, $S_{\tau(\p_i)}$ is a subset of some translate of $A(R'_\p) \subseteq A(-T_i+R_{\p_i})$. Together, it follows that $S_{\tau(\p_i)}=A(\delta_i-T_i+R_{\p_i})$. This contradicts the assumption that $S_{\tau(\p_i)}$ is not a translate of $A(-T_i+R_{\p_i})$, finishing the proof of \ref{dynmorb}.
\qedhere
\end{enumerate}
\end{proof}

We obtain the following criterion for topological conjugacy:

\begin{corollary}\label{thm_conjugacy_criterion}
Let $R$ and $S$ be non-large cofinite sieves for $K$ and $L$, respectively. The topological dynamical systems $\D(K,R)$ and $\D(L,S)$ are topologically conjugate if and only if there is a $\Q$-algebra isomorphism $\tau:K\ra L$ and a unit $\varepsilon\in\O_L^\times$ such that for all primes $\p$ of $K$, the set $S_{\tau(\p)}$ is a translate of $\varepsilon\cdot\tau(R_\p)$.
\end{corollary}

\begin{proof}
The implication ``$\Rightarrow$'' follows immediately from parts \ref{dynmora} and \ref{dynmorc} of \Cref{dynmor}. For ``$\Leftarrow$'' we can construct an isomorphism $(f,A):\D(K,R)\ra\D(L,S)$ as follows: $A(x)=\varepsilon\cdot\tau(x)$ for $x\in\O_K$ and $f(X)=A(X)$ for $X\in\X(K,R)$.
\end{proof}

Given the symmetry group of $\D(K,R)$, we can compute its extended symmetry group:

\begin{corollary}\label{thm_norm_mod_cent}
Let $R$ be a non-large cofinite sieve for $K$ and let $\D=\D(K,R)$ be the corresponding topological dynamical system. Then:
\begin{enumerate}[label=(\alph*)]
\item
The image of the map (\ref{forgetful_map}) is the group $\LinSym(\D)$ of maps $A:\O_K\ra\O_K$ of the form $A=M_\varepsilon\circ\tau$, where $\tau$ is a $\Q$-algebra automorphism of $K$ and $\varepsilon\in\O_K^\times$ is a unit such that $R_{\tau(\p)}$ is a translate of $\varepsilon\cdot\tau(R_\p)$ for all primes $\p$ of $K$.
\item
The surjective group homomorphism $\ExSym(\D) \twoheadrightarrow \LinSym(\D)$, $(f,A)\mapsto A$ splits, so
\[
\ExSym(\D) = \Sym(\D)\rtimes\LinSym(\D).
\]
\end{enumerate}
\end{corollary}

\begin{proof}
This follows by the same argument as \Cref{thm_conjugacy_criterion}.
\end{proof}

\subsection{Application to \texorpdfstring{$k$}{k}-free integers}

We now apply the results from the previous section to the case of $k$-free sieves. For some open questions regarding other sieves, see also \Cref{section_morphism_open_questions}.

The symmetry group of $\D_{K,k}$ was computed in \cite[Theorem~5.3]{positive-entropy-shifts}, so we can show \Cref{intro_automorphisms}:

\begin{theorem}\label{thm_automorphisms}
Let $K$ be an étale $\Q$-algebra and let $k\geq1$.
\begin{enumerate}[label=(\alph*)]
\item The symmetry group of $\D_{K,k}$ is the group of translations by elements of $\O_K$.
\item The extended symmetry group consists of the pairs $(f,A)$, where $A:\O_K\ra\O_K$ is of the form $A=M_\varepsilon\circ\tau$ with $\tau\in\Aut(K)$ and $\varepsilon\in\O_K^\times$, and where $f:\X_{K,k}\ra\X_{K,k}$ is defined by $f(S)=t+\varepsilon\cdot\tau(S)$ with $t\in\O_K$. In particular, the extended symmetry group is isomorphic to $\O_K\rtimes(\O_K^\times\rtimes\Aut(K))$.
\end{enumerate}
\end{theorem}
\begin{proof}
\begin{enumerate}[label=(\alph*)]
\item
We essentially follow \cite[Theorem~5.3]{positive-entropy-shifts}. Let $f$ be any symmetry (so $(\id,f)$ is an isomorphism) and let $M$ and $\T$ as in \Cref{special_chl}\ref{special_chl_block}. By part \ref{special_chl_singleton} of the same lemma, there is a singleton set $T_1=\{t_1\}\in\T$. Now, every (finite!) set $T_2\in\T$ must contain $t_1$: Otherwise, we can choose a prime $\p$ with $t_1\nequiv t_2\mod\p^k$ for all $t_2\in T_2$. Then, $\bigcap_{T\in\T}(-T+R_\p)\subseteq(-t_1+\p^k)\cap(-T_2+\p^k) = \emptyset$ cannot contain a translate of $R_\p=\p^k$, contradicting \Cref{dynmor}\ref{dynmorq}.

Hence, we indeed have $t_1\in T$ for all $T\in\T$. By \Cref{special_chl}\ref{special_chl_block}, it follows that $f(X)\subseteq -t_1+X$ for all $X\in\X_{K,k}$. Applying the same argument to the inverse morphism, we see that there is a number $t_1'\in\O_K$ such that $X\subseteq -t_1'+f(X)$ for all $X\in\X_{K,k}$. In summary, $X \subseteq -t_1'+f(X) \subseteq -t_1-t_1' + X$. Plugging in a singleton set $X$, we learn that $t_1+t_1'=0$. Then, equality follows for all $X\in\X_{K,k}$, so the map $f:\X_{K,k}\ra\X_{K,k}$ is given by translation by $t_1'$.
\item
Clearly, $\varepsilon\cdot\tau(\p^k)=\tau(\p)^k$ for all $\varepsilon\in\O_K^\times$ and $\tau\in\Aut(K)$. The claim follows with \Cref{thm_norm_mod_cent}.
\qedhere
\end{enumerate}
\end{proof}

We can now prove the implication (b) $\Rightarrow$ (c) in \Cref{intro_conjugate}:

\begin{theorem}\label{thm_factor_kfree}
Let $K,L$ be étale $\Q$-algebras and let $k,l\geq1$. If $\D_{L,l}$ is a factor system of $\D_{K,k}$, then $K\cong L$ and $k=l$.
\end{theorem}
\begin{proof}
By \Cref{dynmor}\ref{dynmora}\ref{dynmorb}, there is a $\Q$-algebra isomorphism $\tau:K\ra L$, a unit $\varepsilon\in\O_L^\times$, and a non-empty finite set $T\subseteq\O_K$ such that
\[
\varepsilon\cdot\tau(-T+\p^k)=\tau(\p)^l
\]
for all but finitely many primes $\p$ of $K$. The right-hand side has measure
\[
\meas(\tau(\p)^l) = \meas(\p^l) = \Nm(\p)^{-l}.
\]
The left-hand side has measure
\begin{align*}
&\meas(\varepsilon\cdot\tau(-T+\p^k))
= \meas(-T+\p^k) \\
&\in [\meas(\p^k),|T|\cdot\meas(\p^k)]
= [\Nm(\p)^{-k},|T|\cdot\Nm(\p)^{-k}].
\end{align*}
Comparing the two measures for $\Nm(\q)\ra\infty$, we conclude that $k=l$.
\end{proof}

We now use the local-global principle proved in \Cref{thm_sqfree_local_global} to show the equivalence of the global and the local definition given in \Cref{introduction} of the set $\X_{K,k}$. This was shown (for number fields) in \cite[Theorem~5.3]{positive-entropy-shifts}. To make this paper more self-contained, we reproduce their proof using our notation.

\begin{theorem}\label{thm_orbit_closure}
Let $k\geq2$ and let $R$ be the $k$-free sieve for $K$. Then, $\X_{K,k}=\X(K,R)$ is the closure in $\P(\O_K)$ of the $\O_K$-orbit of the set $V_{K,k}$ of $k$-free numbers in $\O_K$.
\end{theorem}
\begin{remark}
This fails spectacularly for $k=1$, where the closure of the $\O_K$-orbit of $V_{K,k}=\O_K^\times$ is much smaller than $\X_{K,k}$.
\end{remark}
\begin{proof}
It is clear that every $\O_K$-translate of $V_{K,k}$ lies in $\X(K,R)$. As $\X(K,R)$ is a closed subset of $\P(\O_K)$, it therefore contains the closure of the $\O_K$-orbit of $V_k$.

Conversely, we need to prove that every set $X\in\X(K,R)$ lies in this closure. By the definition of the product topology, this means that for every finite subset $M$ of $\O_K$, we need to find a translate $-\Delta+V_{K,k}$ of $V_{K,k}$ such that $(-\Delta+V_{K,k})\cap M=X\cap M$. Let $X'=X\cap M$ and let $\{y_1,\dots,y_r\}=M\setminus X'$ be its complement inside $M$. We need to find some $\Delta\in\O_K$ such that $X'\subseteq -\Delta+V_{K,k}$ and $y_i\notin -\Delta+V_{K,k}$ for $i=1,\dots,r$. The first condition is equivalent to the condition that $\Delta\notin-X'+\p^k$ for all primes $\p$. The second condition follows if for each $i=1,\dots,r$, there is a prime $\p_i$ with $\Delta\in -y_i+\p_i^k$.

As $y_i\notin X'$, we can choose distinct primes $\p_1,\dots,\p_r$ such that $y_i\notin X'+\p_i^k$. Applying \Cref{thm_sqfree_local_global} to the sieve $R'$ given by $R'_\p=-X'+\p^k$, which has $R'_\p\neq\O_{K,\p}$ for all primes $\p$ as $X'\in\X(K,R)$, we see that there is an element $\Delta\in V(K,R')=\O_K\setminus\bigcup_\p(-X'+\p^k)$ with $\Delta\equiv-y_i\mod\p_i^k$. This proves the claim.
\end{proof}

In response to a question from the referee, we prove:

\begin{theorem}
Let $R$ be a non-large sieve for $K$. Then, $\X(K,R)$ is the closure in $\P(\O_K)$ of the $\O_K$-orbit of some (admissible) set $X\in\X(K,R)$.
\end{theorem}
\begin{proof}
Let $(A_1,B_1),(A_2,B_2),\dots$ be an enumeration of all pairs $(A,B)$ of subsets $A\subseteq B\subseteq\O_K$ with $A\in\X(K,R)$ admissible and $B\subseteq\O_K$ finite. We iteratively construct elements $t_1,t_2,\dots$ of $\O_K$ such that the sets $t_i+B_i$ for $i=1,2,\dots$ are pairwise disjoint and such that $\bigcup_{1\leq j\leq i}(t_j+A_j)$ is admissible for all $i=1,2,\dots$. The claim then follows by taking $X := \bigcup_{i\geq1} (t_i+A_i)$: by construction, $X$ is admissible and every admissible set looks locally everywhere like some translate of $X$, so $\X(K,R)$ is the closure of the $\O_K$-orbit of $X$.

Assume we have constructed $t_1,\dots,t_i$ satisfying the conditions. To construct $t_{i+1}$, we apply \Cref{dynamical_crt2} to the sieve $R$, the finite set $M = \bigcup_{1\leq j\leq i}(t_j+B_j)\cup B_{i+1}$, the admissible sets $T_1=\bigcup_{1\leq j\leq i}(t_j+A_j)$ and $T_2=A_{i+1}$, and the open subsets $U_1=U_2=\prod_\p\O_{K,\p}$. We obtain elements $x_1,x_2$ of $\O_K$ such that $(x_1+T_1)\cup(x_2+T_2)$ is admissible and $x_1+M$ and $x_2+M$ are disjoint. Taking $t_{i+1} := x_2-x_1$, we see that $T_1\cup(x_2-x_1+T_2) = \bigcup_{1\leq j\leq i+1}(t_j+A_j)$ is admissible and that the sets $t_j+B_j\subseteq M$ and $t_{i+1}+B_{i+1}\subseteq x_2-x_1+M$ are disjoint for all $1\leq j\leq i$. This completes the induction step.
\end{proof}

\subsection{Open questions}\label{section_morphism_open_questions}

The results from \Cref{section_shift_morphisms} do not completely describe the morphisms between admissible shift spaces $\D(K,R)$, leaving space for future research of a more combinatorial nature, even over the field~$\Q$.

\subsubsection{Symmetry groups}

The proof that the symmetry group of $\D_{K,k}$ consists only of translations relies on the fact that the $k$-free sieve only excludes a single residue class for each prime~$\p$.
The following example shows that, without this restriction, not all elements of the symmetry group of $\D(K,R)$ are necessarily translations:

\begin{example}
Take $K=\Q$ and consider the non-large cofinite sieve given by
\[
R_p =
\begin{cases}
\emptyset,&p=2,3,\\
p\Z_p\cup(1+p\Z_p),&p\geq5.
\end{cases}
\]
Then, there is a symmetry $f:\X(\Q,R)\ra\X(\Q,R)$ given by
\[
f(X) = f^{-1}(X) = \{x\in\Z:\{x-1,x+1\}\subseteq X\Leftrightarrow x\notin X\}\quad\textnormal{for}\quad X\in\X(\Q,R).
\]
(See \Cref{fig_extra_symmetry}. The map $f$ can also be described by the window $M=\{-1,0,1\}$ and patterns $T_1=\{0\}$, $T_2=\{0,1\}$, $T_3=\{-1,0\}$, $T_4=\{-1,1\}$.)
\end{example}

\begin{figure}[h]
\begin{tikzpicture}[scale=0.4]
\def\beg{-4}
\def\en{20}
\newcommand{\drawboxes}[2]{
	\foreach \x in {#2} {
		\fill[lightgray] (\x,0) rectangle ({\x+1},1);
	}
	\foreach \x in {\beg,...,\en} {
		\draw (\x,0) rectangle ({\x+1},1);
	}
	\node[left] at ({\beg},0.5) {$\cdots$};
	\node[right] at ({\en+1},0.5) {$\cdots$};
	\node[left] at ({\beg-2},0.5) {#1};
}
\begin{scope}
	\drawboxes{$X = $}{-3,-2,-1,2,3,4,9,17,19}
\end{scope}
\draw[<->] ({(\beg+\en+1)/2},-0.75) -- node[right] {$f$} +(0,-1.5);
\begin{scope}[shift={(0,-4)}]
	\drawboxes{$f(X) = $}{-3,-1,2,4,9,17,18,19}
\end{scope}
\end{tikzpicture}
\caption{
An admissible set $X$ and its image $f(X)$. (Each square corresponds to an integer. The shaded squares are the elements of $X$ and $f(X)$, respectively.)
}\label{fig_extra_symmetry}
\end{figure}

\begin{remark}
\cite[Section 3.4]{dymek-kasjan-keller-automorphisms} gave interesting examples of $\mathcal B$-free systems with non-trivial symmetry groups (but where the elements of $\mathcal B$ are not pairwise coprime).
\end{remark}

\subsubsection{Factor systems}

\Cref{dynmor}\ref{dynmora}\ref{dynmorb} gives necessary but not sufficient conditions for the shift space $\D(L,S)$ to be a factor system of $\D(K,R)$. Things again get interesting if we allow $S_\p$ to be the union of several residue classes.

Here is an example of two non-isomorphic dynamical systems where one is a factor system of the other:

\begin{example}
Take $K=L=\Q$ and consider the non-large cofinite sieves given by
\begin{align*}
R_p =
\begin{cases}
\emptyset,&p=2,\\
p\Z_p,&p\geq3,
\end{cases}
&&\textnormal{and}&&
S_p =
\begin{cases}
\emptyset,&p=2,\\
p\Z_p\cup(1+p\Z_p),&p\geq3.
\end{cases}
\end{align*}
Then, there is a factor map $(f,A):\D(\Q,R)\ra\D(\Q,S)$ defined by $A=\id$ and
\[f(X)=\{x\in\Z:\{x,x+1\}\subseteq X\}\quad\textnormal{for}\quad X\in\X(\Q,R).
\]
A preimage of $Y\in\X(\Q,S)$ is
\[
\{x\in\Z:x\in Y\textnormal{ or }x-1\in Y\} \in \X(\Q,R).
\]
This factor map satisfies \Cref{dynmor}\ref{dynmorb} with the pattern $T=\{0,1\}$. (See \Cref{fig_factor}.)
\end{example}

\begin{figure}[h]
\begin{tikzpicture}[scale=0.4]
\def\beg{0}
\def\en{24}
\newcommand{\drawboxes}[2]{
	\foreach \x in {#2} {
		\fill[lightgray] (\x,0) rectangle ({\x+1},1);
	}
	\foreach \x in {\beg,...,\en} {
		\draw (\x,0) rectangle ({\x+1},1);
	}
	\node[left] at ({\beg},0.5) {$\cdots$};
	\node[right] at ({\en+1},0.5) {$\cdots$};
	\node[left] at ({\beg-2},0.5) {#1};
}
\begin{scope}[shift={(0,1.5)}]
	\drawboxes{$X_1 = $}{1,6,7,9,12,13,16,18,19,21,22}
\end{scope}
\begin{scope}
	\drawboxes{$X_2 = $}{6,7,12,13,18,19,21,22}
\end{scope}
\draw[->] ({(\beg+\en+1)/2},-0.75) -- node[right] {$f$} +(0,-1.5);
\begin{scope}[shift={(0,-4)}]
	\drawboxes{$f(X_i) = $}{6,12,18,21}
\end{scope}
\end{tikzpicture}
\caption{
Two admissible sets $X_1,X_2$ with the same image $f(X_i)$.
}\label{fig_factor}
\end{figure}

There are also examples of factor maps for which there is an exceptional prime $\p$ not satisfying the condition in \Cref{dynmor}\ref{dynmorb}:

\begin{example}
Take $K=L=\Q$ and consider the non-large cofinite sieves given by
\begin{align*}
R_p =
\begin{cases}
\emptyset,&p=2,3,\\
p\Z_p,&p=5,\\
\{0,1,2\}+p\Z_p,&p\geq7,
\end{cases}
&&\textnormal{and}&&
S_p =
\begin{cases}
\emptyset,&p=2,3,\\
\{0,3\}+p\Z_p,&p=5,\\
\{0,1,2,3,4,5\}+p\Z_p,&p\geq7.
\end{cases}
\end{align*}
There is a factor map $(f,A):\D(K,R)\ra\D(L,S)$ with $A=\id$ and with $f$ defined by the window $M=\{0,1,2,3\}$ and family of patterns $\T=\{T_1,T_2\}$, where $T_1=\{0,1,3\}$ and $T_2=\{0,2,3\}$. If $Y\in\X(\Q,S)$ is disjoint from $\delta_5+S_5$, a preimage of $Y$ is
\[
\bigcup_{\substack{y\in Y:\\y-\delta_5\equiv1\bmod5}}(y+T_1) \cup \bigcup_{\substack{y\in Y:\\y-\delta_5\equiv2,4\bmod5}}(y+T_2) \in \X(\Q,R).
\]
For both patterns $T_i\in\T$, the set $S_5$ is not a translate of $-T_i+R_5$. The surjection $f:\X(K,R)\twoheadrightarrow\X(L,S)$ does not have a continuous section.
\end{example}

\subsection{Patch-counting entropy}

In addition to \Cref{dynmor}\ref{dynmora}\ref{dynmorb}, there is another obstruction to the existence of factor maps given by the \emph{patch-counting entropy}, which coincides with the topological entropy (see \cite[Section 3]{baake-lenz-richard}).

\begin{lemma}
Let $R$ be any sieve for $K$. Choose any $\Q$-vector space norm $\|\cdot\|$ on $K$. Denote the ball of radius $N$ by $B(N)$. The \emph{patch-counting entropy} of $\D(K,R)$ is
\[
E(K,R) := \lim_{N\ra\infty} \frac{\log\#\{E\subseteq\O_K\cap B(N):E\in\X(K,R)\}}{\#(\O_K\cap B(N))} = \log(2)\cdot\prod_\p(1-\meas(R_\p)).
\]
\end{lemma}

This lemma was proved in \cite[Theorem 5.1]{baake-bustos-nickel-power-free-points} for $k$-free sieves with $k\geq2$. Their proof does not quite go through when $k=1$ due to the failure of the local-global principle. We therefore use a slightly different proof for the inequality ``$\geq$''.

\begin{proof}
\begin{description}
\item[``$\leq$'']
Let $R_\p$ be defined modulo $\p^{k(\p)}$. For any $M>1$ and any tuple
\[
\delta=(\delta_\p)_\p\in\prod_{\p:\Nm(\p)\leq M}\O_K/\p^{k(\p)},
\]
define
\[
U(\delta) = \Bigg(\O_K\setminus\bigcup_{\p:\Nm(\p)\leq M}(\delta_\p+R_\p)\Bigg) \cap B(N).
\]
For any $M$, we have
\begin{align*}
\#\{E\subseteq\O_K\cap B(N):E\in\X(K,R)\}
&\leq \sum_{\delta\textnormal{ as above}} \#\{E\subseteq U(\delta)\}
= \sum_{\delta\textnormal{ as above}} 2^{\#U(\delta)},
\end{align*}
By the Chinese remainder theorem,
\[
\lim_{N\ra\infty} \frac{\#U(\delta)}{\#(\O_K\cap B(N))} = \prod_{\p:\Nm(\p)\leq M}(1-\meas(R_\p)).
\]
Since the number of tuples $\delta$ is independent of $N$, we conclude that
\[
\limsup_{N\ra\infty} \frac{\log\#\{E\subseteq\O_K\cap B(N):E\in\X(K,R)\}}{\#(\O_K\cap B(N))}
\leq \log(2)\cdot\prod_{\p:\Nm(\p)\leq M}(1-\meas(R_\p)).
\]
Letting $M$ go to infinity, we get
\[
\limsup_{N\ra\infty} \frac{\log\#\{E\subseteq\O_K\cap B(N):E\in\X(K,R)\}}{\#(\O_K\cap B(N))}
\leq \log(2)\cdot\prod_{\p}(1-\meas(R_\p)).
\]
\item[``$\geq$''] Independently for all primes $\p$, pick a random element $\delta_\p\in\O_{K,\p}$ uniformly according to the $\p$-adic (probability) measure. For any $x\in\O_K$, the probability that $x\notin \delta_\p+R_\p$ holds for all primes~$\p$ is $\prod_\p(1-\meas(R_\p))$. Hence, the expected number of $x\in\O_K\cap B(N)$ with $x\notin \delta_\p+R_\p$ for all $\p$ is $\#(\O_K\cap B(N))\cdot\prod_\p(1-\meas(R_\p))$. In particular, there exist choices $\delta_\p\in\O_{K,\p}$ for all primes $\p$ such that the number of such $x$ is at least $\#(\O_K\cap B(N))\cdot\prod_\p(1-\meas(R_\p))$. Then,
\begin{align*}
&\log \#\{E\subseteq\O_K\cap B(N):E\in\X(K,R)\}
\geq{} \log \#\{E\subseteq\O_K\setminus\bigcup_\p(\delta_\p+R_\p)\} \\
={}& \log(2)\cdot\#(\O_K\setminus\bigcup_\p(\delta_\p+R_\p))
\geq{} \log(2)\cdot \#(\O_K\cap B(N))\cdot\prod_\p(1-\meas(R_\p)).
\qedhere
\end{align*}
\end{description}
\end{proof}

\begin{remark}
If $\D(L,S)$ is a factor system of $\D(K,R)$, then the patch-counting entropies satisfy $E(K,R)\geq E(L,S)$.
\end{remark}

\section{Units}\label{section_units}

The goal of this section is to prove the following statement:

\begin{theorem}\label{units_thm}
Let $K$ be a totally real number field and let $L$ be an étale $\Q$-algebra. Any $\Z$-linear map $A:\O_K\ra\O_L$ with $A(\O_K^\times)\subseteq\O_L^\times$ is of the form $A=M_\varepsilon\circ\tau$ for a $\Q$-algebra homomorphism $\tau:K\ra L$ and a unit $\varepsilon\in\O_L^\times$.
\end{theorem}

Let $K$ be a totally real number field of degree $n$. The embeddings $\sigma_1,\dots,\sigma_n:K\hookrightarrow\C$ combine to an embedding $\sigma:K\hookrightarrow\C^n$. Let
\[
U = \{\varepsilon\in\O_K^\times : \Nm(\varepsilon) = 1\}
\]
and
\[
W = \{x\in\C^n : \Nm(x) = 1\}.
\]
Note that $W$ is an irreducible subvariety of $\C^n$ with coordinate ring
\[
\C[X_1,\dots,X_n]/(X_1\cdots X_n-1) \cong \C[X_1,\dots,X_{n-1},X_1^{-1},\dots,X_{n-1}^{-1}] =: R.
\]
(The isomorphism sends $X_n$ to $X_1^{-1}\cdots X_{n-1}^{-1}$.)

\begin{lemma}
Assume $K$ is a totally real number field. Then, the image $\sigma(U)$ is Zariski dense in $W$.
\end{lemma}
\begin{proof}
Consider any $0\neq f\in R$. For $m=(m_1,\dots,m_{n-1})\in\Z^{n-1}$, write $X^m = X_1^{m_1}\cdots X_{n-1}^{m_{n-1}}$. Let $M$ be the set of $m\in\Z^{n-1}$ for which $f$ has a non-zero $X^m$-coefficient.

Pick a corner $m\in M$ of the convex polytope spanned by $M$ (also called the Newton polytope of the Laurent polynomial $f$). Denote the standard inner product on $\R^{n-1}$ by `$\cdot$'. Then,
\[
C := \{t\in\R^{n-1}: (m-m')\cdot t > 0\textnormal{ for all }m\neq m'\in M\}
\]
is a non-empty open convex cone in $\R^{n-1}$.

By Dirichlet's unit theorem, the image $l(\O_K^\times)$ of $\O_K^\times$ under the map
\[
\begin{tikzcd}[row sep=0]
\O_K^\times \rar{l} & \R^{n-1} \\
\varepsilon \rar[mapsto] & (\log|\sigma_1(\varepsilon)|,\dots,\log|\sigma_{n-1}(\varepsilon)|)
\end{tikzcd}
\]
is a full lattice in $\R^{n-1}$. Since $U$ is a subgroup of finite index of $\O_K^\times$, the image $l(U)$ of $U$ is also a full lattice.

A non-empty open cone always intersects any full lattice. Pick some $\varepsilon\in U$ so that $l(\varepsilon)\in C$. By definition,
\[
(m-m')\cdot l(\varepsilon) > 0\qquad\textnormal{for all }m\neq m'\in M.
\]
It follows that
\[
(m-m')\cdot l(\varepsilon^s) = s\cdot(m-m')\cdot l(\varepsilon)\textnormal{ goes to $\infty$ as $s\ra\infty$} \qquad\textnormal{for all }m\neq m'\in M.
\]
For any monomial $g=X^{m'}\in R$ and any $\zeta\in U$, we have $|g(\zeta)| = e^{m'\cdot l(\zeta)}$. Hence, the quotient of the $X^m$-monomial and any $X^{m'}$-monomial in $f$, both evaluated at $\varepsilon^s$, goes to $\infty$ as $s\ra\infty$. For sufficiently large $s$, the $X^m$-monomial will dominate the sum of all other monomials in $f$, which implies that $f(\varepsilon^s)\neq0$.
\end{proof}

\begin{proof}[Proof of \Cref{units_thm}]
By assumption, for all $x\in\O_K^\times$, we have $A(x)\in\O_L^\times$, so in particular $\Nm(A(x))=\pm1$.

We extend the $\Z$-linear map $A:\O_K\ra\O_L$ to a $\C$-linear map $A:K\otimes\C\ra L\otimes\C$.
Using the isomorphisms $K\otimes\C\cong\C^n$ and $L\otimes\C\cong\C^m$, we can interpret $A:K\otimes\C\ra L\otimes\C$ as a map $A':\C^n\ra\C^m$. Since $U\subseteq\O_K^\times$ is Zariski dense in $W\subseteq\C^n\cong K\otimes\C$ by the previous lemma, we conclude from the first paragraph of the proof that the polynomial
\[
f := (\Nm(A'(X))-1)(\Nm(A'(X))+1)
\]
vanishes on $W$.

We have $A'((\C^\times)^n) \subseteq (\C^\times)^m$: For any $x \in(\C^\times)^n$, pick an $n$-th root $\lambda$ of $\Nm(x)\in\C^\times$. Then, $\Nm(\lambda^{-1}\cdot x)=1$, so $\lambda^{-1}\cdot x$ lies in $W$. We have seen above that this implies that $\lambda^{-1}\cdot A'(x) = A'(\lambda^{-1}\cdot x)$ has norm $\pm1$. In particular, $A'(x)\in(\C^\times)^m$.

Hence, $A((K\otimes\C)^\times)\subseteq(L\otimes\C)^\times$. From \Cref{lemma_norm_implies_b}, we learn that $A$ is of the form $M_\varepsilon\circ\tau$ for some $\Q$-algebra homomorphism $\tau:K\ra L$ and some $\varepsilon\in\O_L$. Since $A(\O_K^\times)\subseteq\O_L^\times$, we in fact have $\varepsilon\in\O_L^\times$.
\end{proof}

\bibliographystyle{alphaurl}
\bibliography{refs.bib}

\end{document}